\documentclass[11pt]{article}
\usepackage{graphicx}
 \usepackage{mathptmx,color}      
\usepackage[titletoc,title]{appendix}
\usepackage[numbers]{natbib}
\usepackage[T1]{fontenc}
\usepackage[applemac]{inputenc}
\usepackage[all]{xy}
\usepackage{enumerate}
\usepackage{comment}
\usepackage{amsmath,amssymb,stmaryrd}
\usepackage[british]{babel}
\usepackage{a4wide}
\usepackage[sans]{dsfont}
\usepackage{amsfonts}
\usepackage{amsthm}
\usepackage[mathscr]{euscript}
\usepackage{graphicx}
\usepackage{mathrsfs}
\usepackage{latexsym}
\usepackage{setspace}
\usepackage{amsthm,amsmath}

\newcommand{\p}{^}

\newcommand{\X}[1]{X\p{#1}}

\newcommand{\pb}[2]{
	\ensuremath{\langle #1,#2 \rangle}}

\newcommand{\wt}{\widetilde}

	\newcommand{\rmd}{\mathrm{d}}
\newcommand{\rme}{\mathrm{e}}
\newcommand{\rmi}{\mathrm{i}}

\newcommand{\M}{\mathrm{M}}

\newcommand{\Lrm}{\mathrm{L}}

\newcommand{\Ws}{\mathrm{W}\p\sig}
\newcommand{\Var}{\mathrm{Var}}

\newcommand{\Hlocz}{\Hscr_{0,\mathrm{loc}}}

\newcommand{\Ascr}{\mathscr{A}}
\newcommand{\Bscr}{\mathscr{B}}
\newcommand{\Cscr}{\mathscr{C}}

\newcommand{\Escr}{\mathscr{E}}
\newcommand{\Fscr}{\mathscr{F}}
\newcommand{\Gscr}{\mathscr{G}}
\newcommand{\Hscr}{\mathscr{H}}

\newcommand{\Kscr}{\mathscr{K}}

\newcommand{\Pscr}{\mathscr{P}}

\newcommand{\Tscr}{\mathscr{T}}

\newcommand{\Xscr}{\mathscr{X}}

\newcommand{\Hloc}{\Hscr_{\mathrm{loc}}}

\newcommand{\zt}{\zeta}
\newcommand{\et}{\eta}

\newcommand{\bt}{\beta}
\newcommand{\ep}{\varepsilon}
\newcommand{\lm}{\lambda}

\newcommand{\Lm}{\Lambda}
\newcommand{\sig}{\sigma}

\newcommand{\om}{\omega}
\newcommand{\Om}{\Omega}
\newcommand{\cadlag}{c\`adl\`ag\ }

\newcommand{\Ebb}{\mathbb{E}}

\newcommand{\Fbb}{\mathbb{F}}

\newcommand{\Nbb}{\mathbb{N}}
\newcommand{\Pbb}{\mathbb{P}}

\newcommand{\Rbb}{\mathbb{R}}

\newcommand{\Hbb}{\mathbb{H}}
\newcommand{\Gbb}{\mathbb{G}}

\newcommand{\aPP}[2]{\ensuremath{\langle #1,#2 \rangle}}

\newtheorem{theorem}{Theorem}[section]
\newtheorem{lemma}[theorem]{Lemma}
\newtheorem{proposition}[theorem]{Proposition}
\newtheorem{corollary}[theorem]{Corollary}
\theoremstyle{definition}
\newtheorem{definition}[theorem]{Definition}
\newtheorem{example}[theorem]{Example}

\newtheorem{assumption}[theorem]{Assumption}

\newtheorem{remark}[theorem]{Remark}

\numberwithin{equation}{section}

\allowdisplaybreaks
\title{Martingale Representation in Progressively Enlarged L\'evy Filtrations}
\author{Paolo Di Tella$^{1,2}$ and Hans-J\"urgen Engelbert $^3$}
\date{}

\begin{document}

%
%
%
%
%

\maketitle
\begin{abstract}{In this paper we obtain a martingale representation theorem in the progressive enlargement $\Gbb$ by a random time $\tau$ of the filtration $\Fbb\p L$ generated by a L\'evy process $L$. The assumptions on the random time are that $\Fbb\p L$ is immersed in $\Gbb$ and that $\tau$ avoids $\Fbb\p L$ stopping times. We also study the multiplicity of a progressively enlarged filtration.}
 \end{abstract}
{\noindent
\footnotetext[1]{Corresponding author.}
\footnotetext[2]{ \emph{Institution:} Inst.\ f\"ur Mathematische Stochastik, TU Dresden, Germany. 

\emph{E-Mail: }{\tt Paolo.Di\_Tella{\rm@}tu-dresden.de}}

\footnotetext[3]{ \emph{Institution:} Department of Mathematics, FSU Jena, Germany.

\emph{E-Mail: }{\tt hans-juergen.engelbert{\rm@}uni-jena.de}}
}

{\noindent \textit{Keywords:}  L\'evy processes, predictable representation property, multiplicity of a filtration, progressive enlargement of filtrations.
}
\\[.5cm]
{\noindent \textit{AMS-Code:}  60H05, 60G46, 60G51; 60H30, 60G44
}

\section{Introduction}\label{sec:intro}
Let $L$ be a L\'evy process and $\Fbb\p L$ the filtration generated by $L$. In this paper, we study the propagation of the predictable representation property (from now on PRP), obtained in \cite[Section 4.2]{DTE15} with respect to the filtration $\Fbb\p L$, to the progressive enlargement $\Gbb$ of $\Fbb\p L$ by a random time $\tau$. The only conditions on the random time $\tau$ we require are that the filtration $\Fbb\p L$ is immersed in $\Gbb$ (also referred to as \emph{hypothesis} $(\Hscr)$) and that $\tau$ avoids $\Fbb\p L$-stopping times (also referred to as \emph{hypothesis} $(\Ascr)$) (see Definition \ref{def.av.im} below). 

The most commonly used definition of the PRP with respect to a filtration $\Fbb$ is: An $\Fbb$-local martingale $X$ possesses the PRP with respect to $\Fbb$ if every $\Fbb$-local martingale can be represented as \emph{stochastic integral} of a \emph{predictable integrand} with respect to $X$. A classical example of a martingale with the PRP is a Brownian motion with respect to its natural filtration.

A first study about the propagation of the PRP to progressively enlarged filtrations is Kusuoka \cite{Ku99}. In \cite{Ku99}, Kusuoka considered a Brownian motion $(W,\Fbb\p W)$ and a random time $\tau$ satisfying hypothesis $(\Hscr)$. Moreover, the $\Gbb$-predictable compensator $\Lm\p\Gbb$ of $\tau$ is assumed to be \emph{absolutely continuous} with respect to the Lebesgue measure. As we shall see below (see Proposition \ref{prop:avoid.bf}), in the Brownian setting, a random time $\tau$ satisfies hypothesis $(\Ascr)$ if and only if $\Lm\p\Gbb$ is continuous. Therefore, our main result, Theorem \ref{thm:PRP.den.hp} below, generalises the result of Kusuoka not only to arbitrary L\'evy processes, but also in the Brownian setting itself, since we only assume the continuity (and not the absolute continuity) of $\Lm\p\Gbb$. 

In Section \ref{subs:sp.c.bf}, our weaker assumptions on $\tau$ allow us to study  the \emph{multiplicity} in the sense of Davis and Varaiya \cite{DV74} (see Definition \ref{def:mult}) of the enlarged Brownian filtration. More precisely, in Theorem \ref{thm:mul.G} we show that, if $\Lm\p\Gbb$ is continuous and singular with respect to the Lebesgue measure (from now on \emph{singular continuous}), then the $\Gbb$-martingale $Z=W+M$, where $M:=1_{[\tau,+\infty)}-\Lm\p\Gbb$, has the PRP with respect to  $\Gbb$. Hence, the multiplicity of $\Gbb$ is equal to one as for $\Fbb\p W$, although $\Gbb$ supports the orthogonal martingales $W$ and $M$ such that $\{W,M\}$ possesses the PRP with respect to $\Gbb$ (see Definition \ref{def:prp} and Theorem \ref{thm:PRP.den.hp}), and, moreover, although $W$ is continuous and $M$ is purely discontinuous. Clearly, this further study of the multiplicity of $\Gbb$ was not possible in \cite{Ku99}, where $\Lm\p\Gbb$ has a density with respect to the Lebesgue measure. 

Our results in Section \ref{subs:sp.c.bf} also generalize previous studies about the multiplicity of an enlarged filtration as the paper \cite{WG82} by Wu and Gang (see, in particular, \cite[Theorem 3]{WG82}) and the work-in-progress by Calzolari and Torti \cite{CT17} (see Section 3 therein).
Indeed, in \cite{WG82} the case of the independent enlargement is considered while in \cite{CT17} the independence is assumed under an equivalent probability measure, called \emph{decoupling measure} in that paper. Contrarily, in the present paper we do not assume the independence between the reference filtration and $\tau$, we only require the hypotheses $(\Hscr)$ and $(\Ascr)$. We notice that hypothesis $(\Hscr)$ is always satisfied if the reference filtration and $\tau$ are independent.

We now give an overview of the literature about the propagation of martingale representation theorems to progressively enlarged filtrations. 

In Aksamit et al.\ \cite{AJR18}, the reference filtration is generated by a Poisson process and enlarged by a general random time $\tau$ to $\Gbb$. Then, a martingale representation theorem for a class of $\Gbb$-martingales stopped in $\tau$ is obtained in the filtration $\Gbb$.  In \cite{DTJ19}, martingale representation theorems in the enlargement of the filtration generated by a point process are studied. In Callegaro et al.\ \cite{CJZ13}, the propagation of PRP to $\Gbb$ is studied for a random time $\tau$ satisfying \emph{Jacod's equivalence hypothesis}. We also recall the work of Coculescu, Jeanblanc and Nikeghbali \cite{CJN12}. Under more general conditions, the propagation of the PRP is studied in \cite{JS15}.

For results about the propagation of the \emph{weak representation property} (i.e., representation with respect to a stochastic integral with respect to a continuous martingale plus a stochastic integral with respect to a compensated jump measure) to a progressively enlarged filtration, we recall Barlow \cite{Bar78}, where the enlargement is obtained by a honest time, and Di Tella \cite{DT19}, where the hypotheses $(\Hscr)$ and $(\Ascr)$ are assumed.

The present paper has the following structure: In Section \ref{sec:setting} we recall some basics which are needed later on. Section \ref{sec:prog.en} is devoted to a review of some results about enlargement of filtrations which are important for this paper. For a comprehensive and extensive presentation of these results we refer to \cite{AJ17}. However, in Section \ref{sec:prog.en} we also give elementary novel proofs of known results. In our opinion, these proofs are of independent interest. In  Section \ref{sec:PRP}, we prove the main result of the present paper about the propagation of the PRP to progressively enlarged L\'evy filtrations. In Section \ref{subs:sp.c.bf}, we study the multiplicity of the progressively enlarged Brownian filtration.

 \section{General Setting}\label{sec:setting}
Let $(\Om,\Fscr,\Pbb)$ be a complete probability space. We consider  a filtration $\Fbb$ satisfying the usual conditions, that is, $\Fbb$ is right-continuous and $\Fscr_0$ contains all $\Pbb$-null sets of $\Fscr$. We define $\Fscr_\infty:=\bigvee_{t\geq0}\Fscr_t$ and $\Fscr_{\infty-}:=\Fscr_\infty$. By $\Pscr(\Fbb)$ we denote the $\sig$-algebra of $\Fbb$-predictable subsets of $\Rbb_+\times\Om$.

For a \cadlag\ process $X$, we denote by $X_-$ the left-limit process of $X$ with the convention $X_{0-}:=X_0$. By $\Delta X:=X-X_-$, we denote the jump process of $X$. Note that $\Delta X_0=0$. We write $X_\infty:=\lim_{t\rightarrow +\infty} X_t$ if the limit on the right-hand side exists and define $X_{\infty-}:=X_\infty$.

For two semimartingales $X$ and $Y$, we denote by $[X,Y]$ the \emph{(quadratic) covariation} of $X$ and $Y$ (see \cite[Definition I.4.45 and Theorem I.4.52]{JS00}).

The space of uniformly integrable martingales $X$ satisfying $\sup_{t\geq 0}\Ebb[X\p2_t]<+\infty$ is denoted by $\Hscr\p2(\Fbb)$. In virtue of \cite[Theorem I.1.42 a)]{JS00}, if $X\in\Hscr\p2(\Fbb)$, then the terminal random variable $X_\infty$ is well defined and, additionally, the inclusion $X_\infty\in L\p2(\Pbb)$ holds. Furthermore, $\|X\|_{\Hscr\p2(\Fbb)}:=\Ebb[X_\infty\p2]\p{1/2}$ defines a norm and the identity $\sup_{t\geq 0}\Ebb[X\p2_t]\p{1/2}=\|X\|_{\Hscr\p2(\Fbb)}$ easily follows. We call the Hilbert space $(\Hscr\p2(\Fbb),\|\cdot\|_{\Hscr\p2(\Fbb)})$ the space of square integrable martingales. By $\Hscr\p2_0(\Fbb)$ we indicate the subspace of $\Hscr\p2(\Fbb)$ consisting of martingales starting at zero a.s. 
By $\Hloc\p2(\Fbb)$ (resp.\ $\Hscr\p2_{0,\mathrm{loc}}(\Fbb)$) we denote the localized version of $\Hscr\p2(\Fbb)$ (resp.\ $\Hscr\p2_0(\Fbb)$).

For $X,Y\in\Hloc\p2(\Fbb)$, $\aPP{X}{Y}$ denotes the \emph{predictable (quadratic) covariation} of $X$ and $Y$ (see \cite[Theorem I.4.2]{JS00}). We say that $X$ and $Y$ are orthogonal if $X_0Y_0=0$ and $\aPP{X}{Y}=0$. In this case, $XY$ is a local martingale. The following lemma will be useful in the sequel.

\begin{lemma}\label{lem:loc.sq.in}
Let $X\in\Hscr_\mathrm{loc}\p2(\Fbb)$ be such that $X_0\in L\p2(\Pbb)$ and $\aPP{X}{X}_\infty\in L\p1(\Pbb)$. We then have $X\in\Hscr\p2(\Fbb)$.
\end{lemma}
\begin{proof}
Let $(\et_n)_n$ be a localizing sequence for $X$. 
Then, by Fatou's Lemma, we get
\[
\Ebb[X\p2_t]\leq\liminf_{n\rightarrow+\infty}\,\Ebb[X_{t\wedge\et_n}\p2]=\Ebb[X_0\p2]+\liminf_{n\rightarrow+\infty}\,\Ebb[\aPP{X}{X}_{t\wedge\et_n}]\leq\Ebb[X_0\p2]+\Ebb[\aPP{X}{X}_\infty]<+\infty,
\]
from which the claim immediately follows. The proof is complete.

\end{proof}

For any bounded or nonnegative measurable process $X$, we denote by ${}\p o X$ the $\Fbb$-\emph{optional projection} and by ${}\p p X$ the $\Fbb$-\emph{predictable projection} of $X$ (see \cite[Th\'eor\`eme V.14 and V.15]{D72}, in \cite[Theorems 5.1 and 5.2]{HWY92}, a more general definition is given). We recall that, for any $\Fbb$-stopping time $\et$, ${}\p o X$ satisfies \[{}\p o X_\et1_{\{\et<+\infty\}}=\Ebb[X_\et1_{\{\et<+\infty\}}|\Fscr_\et]\quad\textnormal{ a.s.}\] Analogously, for every $\Fbb$-predictable stopping time $\et$, ${}\p p X$ satisfies \[{}\p p X_\et1_{\{\et<+\infty\}}=\Ebb[X_\et1_{\{\et<+\infty\}}|\Fscr_{\et-}]\quad\textnormal{ a.s.}\] 

From \cite[Th\'eor\`eme V.20]{D72}, if $X$ is bounded and left-continuous, then ${}\p pX$ is left-continuous as well, while if $X$ is right-continuous, then ${}\p oX$ is c\`adl\`ag.

A process $B$ is called increasing if $B_0=0$, $B_t\in L\p1(\Pbb)$, $t\geq0$, and if the paths $t\mapsto B_t(\om)$ are \emph{increasing} and \emph{right continuous}. Notice that increasing processes are not necessarily adapted but they are c\`adl\`ag. We say that an increasing process $B$ is integrable if its terminal random variable $B_\infty$ belongs to $L\p1(\Pbb)$. For every non-negative measurable process $K$ and every increasing process $B$, we denote by $K\cdot B$ the (Stieltjes--Lebesgue) integral process of $K$ with respect to $B$, that is,
\[
K\cdot B_t(\om):=\int_0\p tK_s(\om)\rmd B_s(\om),\quad (t,\om)\in\Rbb_+\times\Om.
\]

For an increasing process $B$, we denote by $B\p o$ the \emph{dual optional projection} and by $B\p p$ the \emph{dual predictable projection} of $B$ (see \cite[Th\'eor\`eme V.28]{D72}). We recall that $B\p o$ is the  unique $\Fbb$-optional process (resp., $B\p p$ is the unique $\Fbb$-predictable process) such that, for any nonnegative (or bounded) measurable process $X$,
\[
\Ebb[{}\p iX\cdot B_\infty]=\Ebb[X\cdot B\p i_\infty],\quad i=o\ \textnormal{ (resp., } i=p\textnormal{)}.
\]
We remark that $B\p o-B\p p$ is a martingale and, if $B$ is $\Fbb$-adapted, then $B\p o=B$ and $B-B\p p$ is a martingale. In this case, $B\p p$ is called \emph{the} $\Fbb$-\emph{predictable compensator} of $B$.

\section{Progressively Enlarged Filtrations}\label{sec:prog.en}

A \emph{random time} $\tau$ is a $(0,+\infty]$-valued random variable. We introduce the so-called \emph{default process} $H=(H_t)_{t\geq0}$, where $H_t(\om):=1_{[\tau,+\infty)}(t,\om)=1_{\{\tau(\om)\leq t\}}$. Clearly, $H$ is a bounded increasing process and $H_0=0$. By $\Hbb=(\Hscr_t)_{t\geq0}$ we denote the filtration generated by $H$, i.e., for $t\geq0$, we define $\Hscr_t:=\sig(\{H_s,\ s\in[0,t]\})$.

For a filtration $\Fbb$ satisfying the usual conditions, we denote by $\wt\Gbb=(\wt\Gscr_t)_{t\geq0}$ the filtration defined by $\wt\Gscr_t:=\Fscr_t\vee\Hscr_t$, $t\geq0$. The filtration $\Gbb$ is the \emph{smallest} right-continuous filtration containing $\wt\Gbb$, that is,
\begin{equation}\label{eq:def.Gtil}
 \Gscr_t:=\bigcap_{\ep>0} \wt\Gscr_{t+\ep},\quad t\geq0.
\end{equation}
Notice that $\Gbb$ is the smallest right-continuous filtration containing $\Fbb$ and such that $\tau$ is a $\Gbb$-stopping time. Furthermore, since by assumption $\Fscr_0$ contains all $\Fscr$-null sets, $\Gbb$ satisfies the usual conditions.

\begin{definition}\label{def.av.im} Let $\tau$ be a random time.

(i) We say that $\tau$ satisfies  hypothesis $(\Ascr)$ if $\tau$ avoids $\Fbb$-stopping times, that is, if for every $\Fbb$-stopping time $\et$, $\Pbb[\tau=\et<+\infty]=0$ holds. 

(ii) We say that $\tau$ satisfies  hypothesis $(\Hscr)$ if $\Fbb$ is immersed in $\Gbb$, that is, if every $\Fbb$-martingale is also  a $\Gbb$-martingale.
\end{definition}
Hypotheses $(\Ascr)$ and $(\Hscr)$ will play a central role in this paper. Therefore, for the sake of shortness, we formulate the following assumption:
\begin{assumption}\label{ass:en.fi}
\emph{The random time $\tau:\Om\longrightarrow(0,+\infty]$ satisfies the hypotheses $(\Ascr)$ and $(\Hscr)$.}
\end{assumption}

Examples of a random time $\tau$ satisfying Assumption \ref{ass:en.fi} are discussed in Remark \ref{rem:Cox.con} below.

We shall use the following characterization of hypothesis $(\Hscr)$ (see, e.g., \cite[Theorem 3.2]{AJ17}):
\begin{proposition}\label{prop:imm.eq.for}
The three properties below are equivalent:

\textnormal{(i)} Hypothesis $(\Hscr)$ holds.

\textnormal{(ii)} For every $\Gscr_t$-measurable and integrable random variable $G$, the identity $\Ebb[G|\Fscr_\infty]=\Ebb[G|\Fscr_t]$ holds, for every $t\geq0$.

\textnormal{(iii)} For every $\Fscr_\infty$-measurable and integrable random variable $F$, the identity $\Ebb[F|\Gscr_t]=\Ebb[F|\Fscr_t]$ holds, for every $t\geq0$.

\end{proposition}

Let $H\p o$ and $H\p p$ denote the $\Fbb$-dual optional and the $\Fbb$-dual predictable projection of $H$, respectively. We recall the following result:
\begin{lemma}\label{lem:Z.avoi} 
\textnormal{(i)} The identity $\Gscr_0=\Fscr_0$ holds. Hence, $\Gscr_0$ is trivial if and only if $\Fscr_{0}$ is trivial.

\textnormal{(ii)} $H\p o$ is a.s.\ continuous and indistinguishable from $H\p p$ (hence, $H\p p$ is a.s. continuous as well) if and only if $\tau$ satisfies $(\Ascr)$.
\end{lemma}
\begin{proof}
To see (i), we observe that from \cite[Lemme 4.4 a)]{Jeu80} the identity $\Gscr_0=\Fscr_{0}\vee\sig(\{\tau=0\})$ follows. In particular, since $\tau$ takes values in $(0,+\infty]$,  we get $\Gscr_0=\Fscr_0$ and this shows (i). We now come to (ii). First we observe that, if $H\p{o}$ is a.s.\ continuous, then it is clearly $\Fbb$-predictable, since it is  $\Fbb$-adapted. So, $H\p p$ and $H\p o$ are indistinguishable and both continuous. Therefore, it is sufficient to show that $H\p o$ is continuous if and only if $\tau$ satisfies the hypothesis $(\Ascr)$. Let $\et$ be an $\Fbb$-stopping time. The properties of the $\Fbb$-dual optional projection (see \cite[Theorem 5.27]{HWY92}) yield $\Delta H_\eta\p o1_{\{\et<+\infty\}}=\Ebb[\Delta H_\eta|\Fscr_\eta]1_{\{\et<+\infty\}}$ a.s. Hence, taking the expectation in this identity we get
\begin{equation}\label{eq:lem.Z.avoi}
\Ebb\big[\Delta H\p o_\et1_{\{\et<+\infty\}}\big]=\Ebb\big[\Delta H_\eta1_{\{\et<+\infty\}}\big]=\Pbb[\tau=\et<+\infty].
\end{equation}
If we now assume that $H\p o$ is a.s. continuous, then $\Delta H\p o_\et=0$ a.s.\ and \eqref{eq:lem.Z.avoi} implies the hypothesis $(\Ascr)$. Conversely, if the hypothesis $(\Ascr)$ is satisfied, then \eqref{eq:lem.Z.avoi} yields $\Ebb[\Delta H\p o_\et1_{\{\et<+\infty\}}]=0$. But, since $H\p o$ is increasing, we have $\Delta H_\et\p o1_{\{\et<+\infty\}}\ge0$. Hence, we get  $\Delta H\p o_\et1_{\{\et<+\infty\}}=0$ a.s. Thus, for every finite-valued $\Fbb$-stopping time $\et$, we obtain the identity $H\p o_\et=H\p o_{\et-}$ a.s. As a consequence of the optional section theorem (see \cite[Th\'eor\`eme IV.13]{D72}), $H\p o$ and $H\p o_-$ are indistinguishable. In particular, $H\p o$ is continuous. The proof of the lemma is complete.

\end{proof}

 Following \cite[p.63]{Jeu80}, we introduce
\begin{equation}\label{eq:az.supm}
A:={}^o(1-H)={}^o1_{[0,\tau)},
\end{equation}
that is, $A$ is the optional projection of the bounded and \emph{right-continuous} measurable process $(1-H)$. Notice that $A$ is \cadlag. Furthermore,
\begin{equation}\label{eq:def.az.supm}
A_t=\Pbb[\tau>t|\Fscr_t],\quad\textnormal{a.s., }\ t\in[0,+\infty),
\end{equation}
and $A$ is a supermartingale of \emph{class} $(D)$, called \emph{the Az\'ema supermartingale}. We observe that $A_0=1$ and we define $A_\infty:=\Pbb[\tau=+\infty|\Fscr_\infty]$. 

\begin{lemma}\label{lem:prop.Z}
\textnormal{(i)} For any $\Fbb$-stopping time $\et$, we have
\begin{equation}\label{eq:Z.st.t}
A_\et=\Pbb[\tau>\et|\Fscr_\et]1_{\{\et<+\infty\}}+A_\infty1_{\{\et=+\infty\}}\quad \textnormal{a.s.}
\end{equation}

\textnormal{(ii)} The processes $A_-$ and the $\Fbb$-predictable projection ${}\p p1_{[0,\tau]}$ are indistinguishable and, for any $\Fbb$-predictable stopping time $\et$, we have
\begin{equation}\label{eq:Z-.st.t}
A_{\et-}=\Pbb[\tau\geq\et|\Fscr_{\et-}]1_{\{\et<+\infty\}}+A_\infty1_{\{\et=+\infty\}}\quad \textnormal{a.s.}
\end{equation}
\end{lemma}
\begin{proof}
We first observe that it is enough to show \eqref{eq:Z.st.t} and \eqref{eq:Z-.st.t} on $\{\et<+\infty\}$, since, because of $A_{\infty-}=A_\infty$, the formulas
\[
A_\et=A_\et1_{\{\et<+\infty\}}+A_\infty1_{\{\et=+\infty\}}, \quad\textnormal{and}\quad A_{\et-}=A_{\et-}1_{\{\et<+\infty\}}+A_\infty1_{\{\et=+\infty\}}
\]
obviously hold. Therefore, to see (i) it is enough to use the definition of the optional projection of $1_{[0,\tau)}$. We now come to (ii). We show that $A_-$ and ${}\p p1_{[0,\tau]}$ are indistinguishable. From this and the properties of the predictable projection we immediately get \eqref{eq:Z-.st.t} on $\{\et<+\infty\}$. Let us therefore consider a finite-valued $\Fbb$-predictable stopping time $\et$ and an announcing sequence $(\et_n)_{n\geq1}$ for $\et$. In view of (i), we have $A_{\et_n}=\Ebb[\tau>\et_n|\Fscr_{\et_n}]$, for every $n\geq1$. Taking now the limit as $n\rightarrow+\infty$ in this latter identity, using that $\Fscr_{\et-}=\bigvee_{n\geq1}\Fscr_{\et_n}$ (see \cite[Th\'eor\`em III.35]{D72}), from \cite{Me69} and the properties of ${}\p p1_{[0,\tau]}$, we get
\[
A_{\et-}=\Ebb[\tau\ge\et|\Fscr_{\et-}]={}\p p1_{[0,\tau]}(\et),\quad\text{a.s.}
\]
Making use of \cite[Th\'eor\`em IV.13]{D72}, we deduce that $A_-$ and ${}\p p1_{[0,\tau]}$ are indistinguishable. The proof of (ii) is complete. The proof of the lemma is complete.
\end{proof}

It is well known that the Az\'ema supermartingale $A$ is strictly positive on $[0,\tau)$ and that $A_-$ is strictly positive on $[0,\tau]$, see, e.g., \cite[Lemme 4.3 and the paragraph before \S IV.2 on p.\ 63]{Jeu80}. 
On the basis of formulas \eqref{eq:Z.st.t} and \eqref{eq:Z-.st.t}, we  now give an elementary new proof, in our opinion of its own interest, of these properties of the Az\'ema supermartingale.
\begin{theorem}\label{thm:nonneg.A}
Let $\tau$ be a random time and let $A$ be the associated Az\'ema supermartingale. We then have

\textnormal{(i)} \ $A>0$ a.s. on $[0,\tau)$.

\textnormal{(ii)} \ $A_->0$ a.s. on $[0,\tau]$.  
\end{theorem}
\begin{proof}
For verifying (i), we define  the $\Fbb$-stopping time $\zt:=\inf\{t\geq0: A_{t}=0\}$ and observe that, by the right-continuity of $A$, the identity $A_\zt=0$ a.s.\ on $\{\zt<+\infty\}$ holds.  So, from  \eqref{eq:Z.st.t}, we get $\Pbb[\{\tau>\zt\}\cap\{\zt<+\infty\}]=0$, i.e., 
\begin{equation}\label{eq:tau.le.zt}
\tau\leq\zt, \quad\textnormal{a.s.\ on}\ \{\zt<+\infty\}.
\end{equation}
This means that $A_t(\om)>0$ for every $t<\tau$ on $\{\zt<+\infty\}$ a.s. On the other side, we obviously have $A_t(\om)>0$ for every $t<\tau$ on $\{\zt=+\infty\}$ a.s. Hence, $A>0$ on $[0,\tau)$ a.s.\ and the proof of (i) is complete. We now show (ii). For this, we define the sequence $(\zt_n)_{n\geq1}$ of $\Fbb$-stopping times and the random set $C$ by 
\[
\textstyle \zt_n:=\inf\{t>0: A_t\le\frac1n\},\quad n\geq1, \qquad C:=\{\zt_k<\zt,\ \text{ for all }\ k\geq1\}.
\]
The set $C$ is $\Fscr_\zt$-measurable. Indeed, we have $\zt_n\uparrow \zt$, hence $\zt_n$ is $\Fscr_\zt$-measurable, since $\zt_n$ is $\Fscr_{\zt_n}$-measurable and, because of $\zt_n\leq\zt$, $\Fscr_{\zt_n}\subseteq\Fscr_\zt$ (see \cite[Th\'eor\`eme III.28 and III.32]{D72}). So, from  $C=\bigcap_{n\geq1}\{\zt_n<\zt\}$, we get $C\in\Fscr_\zt$. Hence, 
\[
\zt\p\prime:=\begin{cases}\zt, &\textnormal{on }\ C,\\+\infty,&\textnormal{otherwise }
\end{cases}
\]
is an $\Fbb$-stopping time. Furthermore, $\zt\p\prime$ is $\Fbb$-predictable. Indeed, $\zt_n<\zt\p\prime$ on $\{\zt\p\prime<+\infty\}$ (notice that $\zt>0$ a.s., $A$ being c\`adl\`ag and because of $A_0=1$) and $\rho_n:=\zt_n\wedge n$ is an announcing sequence for $\zt\p\prime$. Because of $A_{\zt\p\prime-}1_{\{\zt<+\infty\}\cap C}=\lim_{n\rightarrow+\infty}A_{\rho_n}1_{\{\zt<+\infty\}\cap C}=0$, we obtain the identity $A_{\zt-}=0$, on the set $\{\zt\p\prime<+\infty\}=\{\zt<+\infty\}\cap C$. Applying \eqref{eq:Z-.st.t}, from the definition of $\zt\p\prime$, we get $\tau<\zt$ on $\{\zt<+\infty\}\cap C$. We now use the well-known identity $\zt=\tilde\zt$, where $\tilde\zt:=\inf\{t\geq0: A_{t-}=0\}$ (see, e.g., \cite[Theorem 2.62 and p.\ 63]{HWY92}) to get
\begin{equation}\label{eq:A-.pos.*}
A_{t-}(\om)>0, \quad \textnormal{ for every }\ t\leq\tau\ \textnormal{ on }\ \{\zt<+\infty\}\cap C\ \textnormal{a.s.}
\end{equation}
Next we consider the set $C\p c=\{\exists\ k: \zt_k=\zt\}$. Then, on $[0,\zt]\cap \big(\{\zt<+\infty\}\cap C\p c\times\Rbb_+\big)$, we have $A_{t-}(\om)>0$ for all $t$ a.s. From \eqref{eq:tau.le.zt}, we get
\begin{equation}\label{eq:A-.pos.**}
A_{t-}(\om)>0, \quad \textnormal{ for every }\ t\leq\tau\ \textnormal{ on }\ \{\zt<+\infty\}\cap C\p c\ \textnormal{a.s.}
\end{equation}
Putting \eqref{eq:A-.pos.*} and \eqref{eq:A-.pos.**} together, we obtain 

\begin{equation}\label{eq:A.min.pos.zt.fin}
A_{t-}(\om)>0\quad \text{for every }\ t\leq\tau\ \text{ on }\ \{\zt<+\infty\}\ \text{a.s.}
\end{equation}
Now we pass on to the set $\{\zt=+\infty\}$. First we observe that,
\begin{equation}\label{eq:tau.less.int.Ainf}
\tau<+\infty\quad\ \textnormal{ on }\ \{A_\infty=0\}\cap\{\zt=\infty\}\ \text{ a.s.}
\end{equation}
Indeed,
\[
\begin{split}
\Pbb\big[\{\tau=+\infty\}\cap\{ A_\infty=0\}\cap\{ \zt=+\infty\}\big]&=\Ebb\Big[\Pbb\big[\{\tau=+\infty\}\cap\{ A_\infty=0\}\cap\{ \zt=+\infty\}|\Fscr_\infty\big]\Big]
\\&=
\Ebb\Big[\Pbb\big[\tau=+\infty|\Fscr_\infty\big]1_{\{A_\infty=0\}\cap\{\zt=+\infty\}}\Big]
\\&=
\Ebb\big[A_\infty1_{\{A_\infty=0\}\cap\{\zt=+\infty\}}\big]=0.
\end{split}
\]
Again using $\zt=\tilde\zt$, from \eqref{eq:tau.less.int.Ainf} we get
\begin{equation}\label{eq:A-.pos.***}
A_{t-}(\om)>0, \quad \textnormal{ for every }\ t\leq\tau\ \textnormal{ on }\ \{\zt=+\infty\}\cap \{A_\infty=0\}\ \textnormal{a.s.}
\end{equation}
Finally, we consider the set $\{\zt=+\infty\}\cap \{A_\infty>0\}$. Using the identity $\zt=\tilde \zt$, we conclude that $A_->0$ on $[0,\zt)\cap\big(\{\zt=+\infty\}\cap\{A_\infty>0\}\times\Rbb_+\big)$ a.s. and, since $A_{\infty-}=A_\infty$, we deduce $A_->0$ on $[0,\zt]\cap\big(\{\zt=+\infty\}\cap\{A_\infty>0\}\times\Rbb_+\big)$ a.s. But this implies
\begin{equation}\label{eq:A-.pos.****}
A_{t-}(\om)>0, \quad \textnormal{ for every }\ t\leq\tau\ \textnormal{ on }\ \{\zt=+\infty\}\cap \{A_\infty>0\}\ \textnormal{a.s.}
\end{equation}
Summing up \eqref{eq:A-.pos.***} and \eqref{eq:A-.pos.****}, we find that $A_{t-}(\om)>0$ for every  $t\leq\tau$ on $\{\zt=+\infty\}$ a.s. Finally, from \eqref{eq:A.min.pos.zt.fin}, we get $A_->0$ on $[0,\tau]$ a.s. The proof of the theorem is complete.
\end{proof}

Relying again on formulas \eqref{eq:Z.st.t} and \eqref{eq:Z-.st.t}, we are now going to show some other useful properties of the Az\'ema supermartingale as, e.g., the well-known fact that if $\tau$ satisfies Assumption \ref{ass:en.fi}, then $A$ is continuous and decreasing (see, for example, \cite[Proposition 3.9]{AJ17} for a proof which makes use of the Doob--Meyer decomposition of $A$).  

\begin{theorem}\label{thm:con.Z}
Let $\tau$ satisfy $(\Hscr)$. Then the following properties hold: 

\textnormal{(i)} The Az\'ema supermartingale $A$ is decreasing and, for every $\Fbb$-stopping time $\et$,  it follows
\[
A_{\et-}=\Pbb[\tau\geq\et|\Fscr_{\et}]1_{\{\et<+\infty\}}+A_\infty1_{\{\et=\infty\}}\quad \textnormal{a.s.}
\]

\textnormal{(ii)} The identity $H\p o=1-A$ holds.

\textnormal{(iii)} The random time $\tau$ satisfies $(\Ascr)$ if and only if $A$ is continuous.

\end{theorem}
\begin{proof}
 We first verify (i). The equivalence of condition (ii) of Proposition \ref{prop:imm.eq.for} to $(\Hscr)$ implies that $A$ is decreasing.
Let us now define $\tau_n:=\tau+1/n$, $H\p n:=1_{[\tau_n,+\infty)}$ and $A\p n:={}\p o(1-H\p n)$. Because of hypothesis $(\Hscr)$, from Proposition \ref{prop:imm.eq.for} (see condition (ii)), we deduce 
\[
A\p n_t=\Pbb[\tau_n>t|\Fscr_t]=\Pbb[\tau>t-1/n|\Fscr_\infty]=\Pbb[\tau>t-1/n|\Fscr_{t-1/n}]=A_{t-1/n},\quad \textnormal{a.s., \ }\ t\geq1/n. 
\] 
This means that $(A\p n_t)_{t\geq1/n}$ and $(A_{t-1/n})_{t\geq1/n}$ are modifications. Since they are both c\`adl\`ag, they are indistinguishable and therefore
\[
\lim_{n\rightarrow+\infty}A\p n_t=A_{t-},\quad t\geq0,\quad \textnormal{a.s.}
\]
Applying now \eqref{eq:Z.st.t} to $\tau_n$, we find that $A\p n_\et=\Pbb[\tau_n>\et|\Fscr_\et]1_{\{\et<+\infty\}}+A_\infty1_{\{\et=+\infty\}}$, for every $\Fbb$-stopping time $\et$. Passing to the limit as $n\rightarrow+\infty$, we get (i). We now show (ii). Claim (i) yields that the $\Fbb$-optional process $1-A$ is increasing and that, for every bounded $\Fbb$-stopping time $\et$, the identity
\begin{equation}\label{eq:cor.con.Z.aux1}
\Ebb\big[1_{[0,\et)}\cdot(1-A)_\infty\big]=1-\Ebb[A_{\et-}]=\Pbb[\tau<\et]
\end{equation}
holds. On the other side, from the properties of the $\Fbb$-dual optional projection, we have  
\begin{equation}\label{eq:cor.con.Z.aux2}
\Ebb\big[1_{[0,\et)}\cdot H\p o_\infty\big]=\Pbb[\tau<\et].
\end{equation}
Therefore, comparing \eqref{eq:cor.con.Z.aux1} and \eqref{eq:cor.con.Z.aux2}, we obtain $\Ebb\big[1_{[0,\et)}\cdot(1-A)_\infty\big]=\Ebb\big[1_{[0,\et)}\cdot H\p o_\infty\big]$. The system $\Cscr:=\{1_{[0,\et)},\ \et\ \textnormal{ bounded }\ \Fbb\textnormal{-stopping time}\}$ is stable under multiplication and generates the $\Fbb$-optional $\sig$-algebra. An application of the monotone class theorem (\cite[Theorem 1.4]{HWY92}) yields
\[
\Ebb\big[K\cdot(1-A)_\infty\big]=\Ebb\big[K\cdot H\p o_\infty\big]
\] 
for any $\Fbb$-optional and bounded $K$. Because of the uniqueness of the $\Fbb$-dual optional projection, we get that $(1-A)$ is indistinguishable from $H\p o$. To conclude, we observe that (iii) immediately follows from the previous statement (ii) and from Lemma \ref{lem:Z.avoi} (ii).
The proof of the theorem is complete.
\end{proof} 
\begin{remark}\label{rem:Cox.con}
We remark that by Theorem \ref{thm:con.Z} (iii) it is easy to construct a random time $\tau$ satisfying Assumption \ref{ass:en.fi}. Indeed, to construct a random time  $\tau$ satisfying hypothesis $(\Hscr)$ one can proceed following the Cox construction as shown in, e.g., \cite[Section 6.5]{BR13}. For such a random time $\tau$, the identity $\Pbb[\tau\leq t|\Fscr_t]=\Pbb[\tau\leq t|\Fscr_\infty]$ holds and this, according to \cite[Lemma 3.8]{AJ17}, is equivalent to $(\Hscr)$. Since in \cite[Section 6.5]{BR13} it is shown that the Az\'ema supermartingale $A$ (denoted by $1-F$ therein) is continuous, we see from Theorem \ref{thm:con.Z} (iii) that $\tau$ satisfies Assumption \ref{ass:en.fi}.

Analogously, if $\tau$ is independent of $\Fbb$ and the distribution function $F_\tau$ of $\tau$ is continuous, then $\tau$ satisfies Assumption \ref{ass:en.fi}.
\end{remark}
We also have the following corollary to Theorem \ref{thm:con.Z}:
\begin{corollary}\label{cor:con.Z}
Let $\tau$ satisfy Assumption \ref{ass:en.fi}. 

\textnormal{(i)} The Az\'ema supermartingale $A$ is decreasing, continuous and satisfies $A>0$ on $[0,\tau]$.

\textnormal{(ii)} The Doob--Meyer decomposition of the supermartingale $A$ is $A=1-(1-A)$.

\textnormal{(iii)} The identities $(1-A)=H\p o=H\p p$ hold.

\end{corollary}
\begin{proof}
To verify (i), we notice that Theorem \ref{thm:con.Z} (i) and (iii) imply that $A$ is decreasing and continuous. Thus, since $A=A_-$ holds, from Theorem \ref{thm:nonneg.A} (ii) we immediately deduce $A>0$ on $[0,\tau]$. 
We now verify (ii). In view of Theorem \ref{thm:con.Z}, the $\Fbb$-adapted process $(1-A)$ is increasing and continuous. So $(1-A)$ is a predictable increasing process and this shows (ii). Finally, we observe that the first equality in (iii)  is Theorem \ref{thm:con.Z} (ii) and the second one follows from Theorem \ref{thm:con.Z} (iii) and Lemma \ref{lem:Z.avoi} (ii). The proof of the corollary is complete. 
\end{proof}

The process $H=1_{[\tau,+\infty)}$ is $\Gbb$-adapted and increasing: Let $\Lm\p\Gbb$ denote the $\Gbb$-predictable compensator of $H$. Due to the boundedness of $H$, the increasing process $\Lm\p\Gbb$ is integrable. We introduce the process $M=(M_t)_{t\geq0}$ defined by
\begin{equation}\label{eq:def.M}
M_t:=H_t-\Lm\p\Gbb_t,\quad t\geq0.
\end{equation}
Then, $M$ is a $\Gbb$-local martingale and, additionally, $\sup_{t\geq0}|M_t|\leq 1+\Lm\p\Gbb_\infty\in L\p1(\Gscr_\infty,\Pbb)$ holds. Therefore, $M$ is bounded in $L\p1(\Gscr_\infty,\Pbb)$ and, hence, $M$ is a uniformly integrable $\Gbb$-martingale. We observe that, being of finite variation, $M$ is a purely-discontinuous martingale.  
\\[.2cm]
\indent
From \cite[Remarque 3.5 (3)]{Jeu80}, we obtain the representation
\begin{equation}\label{eq:G.com.gen}
\Lm\p\Gbb_t=\int_0\p{\tau\wedge t}\frac{1}{A_{s-}}\,\rmd H\p p_s,\quad t\ge0.
\end{equation}
Observe that, because of Theorem \ref{thm:nonneg.A} (ii), the right hand side of \eqref{eq:G.com.gen} is well defined.
\\[.2cm]
\indent In the next theorem, by $Y:=\Escr(-M)$ we denote the stochastic exponential (see \cite[\S I.4f]{JS00}) of the martingale $-M$. 

\begin{theorem}\label{thm:G.com.AH}
Let $\tau$ satisfy Assumption \ref{ass:en.fi}.

\textnormal{(i)} $\Lm\p\Gbb$ is continuous and $\Lm_t\p\Gbb=-\log( A_{\tau\wedge t})$, $t\ge0$, with the convention $\log(0):=-\infty$.

\textnormal{(ii)} $M\in\Hscr_0\p2(\Gbb)$ and $\pb{M}{M}=\Lm\p{\Gbb}$.

\textnormal{(iii)} $Y\in\Hscr\p2_\mathrm{loc}(\Gbb)$ and $Y=A\p{-1}1_{[0,\tau)}$.
\end{theorem}
\begin{proof}
The claim (i) immediately follows from \eqref{eq:G.com.gen} and Corollary \ref{cor:con.Z} (i), (iii). We now come to (ii). It is clear that $M_0=0$.  The continuity of $\Lm\p\Gbb$ yields that the uniformly integrable martingale $M$ has a unique jump of size $1$ in $\tau$. So, $M$ is locally bounded and hence $M\in\Hscr\p2_{\mathrm{loc},0}(\Gbb)$. Furthermore, recalling that $M$ is purely discontinuous, we get
\[
[M,M]_t=\sum_{s\leq t}(\Delta M_s)\p2=\sum_{s\leq t}(\Delta H_s)=H_t,\quad t\geq0.
\]
Hence, $[M,M]$ is bounded and, therefore, integrable. This implies the integrability of $\aPP{M}{M}$, which is the $\Gbb$-predictable compensator of the $\Gbb$-adapted increasing process $[M,M]$. From Lemma \ref{lem:loc.sq.in}, we get $M\in\Hscr\p2_0(\Gbb)$. Integration by parts yields $M\p2=2M_-\cdot M+H$. Since the processes $2M_-\cdot M$ and $M\p 2-\aPP{M}{M}$ are $\Gbb$-local martingales, we get that $H-\aPP{M}{M}$ is a $\Gbb$-local martingale. The $\Gbb$-predictable compensator of an increasing process being unique, we deduce the identity $\pb{M}{M}=\Lm\p{\Gbb}$, which completes the proof of (ii). We now verify (iii). Because of the Dol\'eans-Dade equation, from (ii) we derive $Y\in\Hscr\p2_\mathrm{loc}(\Gbb)$. Thanks to Theorem \ref{thm:con.Z} (ii), we have $A>0$ on $[0,\tau]$ so that $A\p{-1}$ is well-defined over $[0,\tau]$. Using the Dol\'eans-Dade exponential formula, we compute
\[\begin{split}
Y_t:=\Escr(-M)_t&=\exp(-M_t)\prod_{0\leq s\leq t}(1-\Delta M_s)\exp(\Delta M_s)\\&=\begin{cases}\exp(-\log( A_{t\wedge\tau})),\ &\textnormal{on } \{t<\tau\},\\
0,\ &\textnormal{else}.
\end{cases}
\end{split}
\]
The proof of the theorem is complete.
\end{proof}

\section{Martingale Representation for Enlarged L\'evy Filtrations }\label{sec:PRP}
In this section we prove a martingale representation theorem in the progressive enlargement of a L\'evy filtration by a random time $\tau$ satisfying Assumption \ref{ass:en.fi}. This extends \cite{DTE15} and \cite{DTE16}, where a martingale representation theorem was obtained for a (non-enlarged) L\'evy filtration, to enlarged filtrations. 

For the sake of simplicity of representation, in what follows we shall restrict ourselves to the study of the predictable representation property for countable, mutual orthogonal families of locally square integrable local martingales. 

\subsection{Definition of the Predictable Representation Property}\label{subs:prp}
Let $\Fbb$ be a filtration satisfying the usual conditions and $X\in\Hscr_\mathrm{loc}\p2(\Fbb)$. We define
\[
\Lrm\p2(\Fbb,X):=\{K:\ K\textnormal{ is }\Fbb\textnormal{-predictable and }\ \Ebb[K\p2\cdot\aPP{X}{X}_{\infty}]<+\infty\}
.\]
Notice that $\Lrm\p2(\Fbb,X)$ is non empty: Indeed, for each $\Fbb$-stopping time $\sig$ such that $X\p\sig\in\Hscr\p2(\Fbb)$, we have $1_{[0,\sig]}\in\Lrm\p2(\Fbb,X)$. The space $\Lrm\p2_\mathrm{loc}(\Fbb,X)$ is defined in the following way: An $\Fbb$-predictable process $K$ belongs to  $\Lrm\p2_\mathrm{loc}(\Fbb,X)$ if and only if $1_{[0,\et_n]}K\in\Lrm\p2(\Fbb,X)$, where $(\et_n)_n$ is an increasing sequence of $\Fbb$-stopping times such that $\et_n\uparrow+\infty$ as $n\rightarrow+\infty$.

For $K\in \Lrm\p2(\Fbb, X)$, the stochastic integral of $K$ with respect to $X$ is denoted by $K\cdot X$ or $\int_0\p\cdot K_s\rmd X_s$ and it is characterized as the unique $Z\in\Hscr_{0}\p2(\Fbb)$ such that $\aPP{Z}{Y}=K\cdot\aPP{X}{Y}$, for every $Y\in\Hscr\p2(\Fbb)$. Notice that, if $K\in\Lrm\p2_\mathrm{loc}(\Fbb,X)$, then $K\cdot X\in\Hlocz\p2(\Fbb)$. 
 
\begin{definition}\label{def:prp}
Let the family $\Xscr=\{X\p n,\ n\geq1\}\subseteq\Hscr_{0,\mathrm{loc}}\p2(\Fbb)$ be such that $\aPP{X\p n}{X\p m}=0$, for $n\neq m$ (i.e., $\Xscr$ consists of pairwise orthogonal martingales). We say that $\Xscr$ possesses the PRP with respect to $\Fbb$, if every $Y\in\Hscr\p2(\Fbb)$ can be written as
\[
Y_t=Y_0+\sum_{n=1}\p{+\infty}K\p n\cdot X\p n_t,\quad X\p n\in\Xscr,\quad K\p n\in\Lrm\p2(\Fbb,X\p n),\quad n\geq1,\quad t\geq0,
\]
where the orthogonal sum converges in $(\Hscr\p2_0(\Fbb),\|\cdot\|_{\Hscr\p2(\Fbb)})$.
\end{definition}
 \subsection{The Predictable Representation Property for L\'evy Processes}\label{subs:Lev.prp}
A L\'evy process with respect to $\Fbb$ is an $\Fbb$-adapted stochastically continuous process $L$ such that  $L_0=0$, $(L_{t+s}-L_t)$ is distributed as $L_s$ and $(L_{t+s}-L_t)$ is independent of $\Fscr_t$, $s,t\geq0$. In this case we say that $(L,\Fbb)$ is a L\'evy process. As $\Fbb$ satisfies the usual conditions,  we can also assume that $L$ is c\`adl\`ag. By $\Fbb\p L$ we denote the completion in $\Fscr$ of the filtration generated by $L$. It is well-known (see \cite{W81}) that $\Fbb\p L$ satisfies the usual conditions. Furthermore, $\Fscr\p L_0(=\Fscr\p L_{0+})$ is trivial and $(L,\Fbb\p L)$ is a L\'evy process.  

We denote by $\mu$ the jump measure of $L$ and recall that $\mu$ is a homogeneous Poisson random measure with respect to $\Fbb\p L$ (see \cite{JS00}, Definition II.1.20). The predictable compensator of $\mu$ is deterministic and given by $\ell_+\otimes\nu$, where $\ell_+$ is the Lebesgue measure on $\Rbb_+$ and $\nu$ is the L\'evy measure of $L$, that is, a $\sig$-finite measure such that $\nu(\{0\})=0$ and $x\mapsto x\p2\wedge1$ is integrable. We write $\overline\mu:=\mu-\ell_+\otimes\nu$ for the compensated Poisson random measure associated with $\mu$. By $\Ws$ we denote the Gaussian part of $L$, which is an $\Fbb\p L$-Brownian motion such that $\Ebb[(\Ws_t)\p2]=\sig\p2 t$, where $\sig\p2\geq0$. By $\bt\in\Rbb$ we denote the drift parameter of $L$ and we call $(\bt,\sig\p2,\nu)$ the $\Fbb$-characteristics of $L$. For every $t\ge0$, the characteristic function of $L_t$ is given by $\Ebb[\exp(\rmi u L_t)]=\exp(t\psi(u))$ where
\[
\psi(u):=\rmi \bt u-\frac{1}2u\p2\sig\p2+\int_\Rbb(\rme\p{\rmi u x}-1-\rmi ux1_{\{|x|\leq1\}})\nu(\rmd x),\quad u\in\Rbb.
\]
The following lemma is a characterization of a L\'evy process with respect to $\Fbb$ and it will be useful later. The proof is straightforward and therefore omitted.
\begin{lemma}\label{char.lev.proc}
Let $L$ be an $\Fbb$-adapted \cadlag\ process starting at zero. Then $(L,\Fbb)$ is a L\'evy process with $\Fbb$-characteristics $(\bt,\sig\p2,\nu)$ if and only
$
Z_t(u):=\exp(\rmi u L_t-t\psi(u))
$ is a complex-valued $\Fbb$-martingale, for every $u\in\Rbb$. 
\end{lemma}

Let $\Gscr\p2_{\mathrm{loc}}(\mu)$ denote the linear space of the $\Bscr(\Rbb_+)\otimes\Pscr(\Fbb)$-measurable mappings $W$ such that the increasing process
$\sum_{s\leq \cdot}G\p2(s,\om,\Delta L_s(\om))1_{\{\Delta L_s(\om)\neq0\}}$ is locally integrable.

For $G\in\Gscr\p2_{\mathrm{loc}}(\mu)$, the stochastic integral $\int_{\Rbb_+\times\Rbb}W(s,x)\overline\mu(\rmd s,\rmd x)$ of $G$ with respect to $\overline\mu$ is defined as the unique purely discontinuous martingale $Z\in \Hscr\p2_{0,\mathrm{loc}}(\Fbb)$ such that 
\[
\Delta Z_t(\om)=G(t,\om,\Delta L_t(\om))1_{\{\Delta L_t(\om)\neq0\}}, \quad t\ge 0
\] 
(see \cite[III.3.\S a]{J79}). The existence and the uniqueness of $Z$ is a consequence of \cite[Theorem 2.45]{J79}.

For any $f\in L\p{\,2}(\nu)$ and $t\geq0$, the deterministic function $G_f(s,x):=1_{[0,t]}(s)f(x)$ belongs to $\Gscr_{\mathrm{loc}}(\mu)$. Indeed, 
\[
\Ebb\left[\sum_{s\leq t}G_f\p2(s,\Delta L_s)1_{\{\Delta L_s\neq0\}}\right]=\Ebb\left[\int_{\Rbb_+\times\Rbb}G\p2_f(s,x)\mu(\rmd s,\rmd x)\right]=t\int_\Rbb f\p2(x)\nu(\rmd x)<+\infty.
\] 
Thus, we can introduce the process $\X{f}\in\Hscr_{0,\mathrm{loc}}\p2(\Fbb)$ by
\begin{equation}\label{eq:en.mar}
\X{f}_{\,t}:=\int_{[0,t]\times\Rbb}f(x)\overline\mu(\rmd s,\rmd x),\quad t\geq0.
\end{equation}

The  properties of $X\p f$, $f\in L\p2(\nu)$, are summarised by the next theorem.

\begin{theorem}\label{thm:en.mar}
Let $(L,\Fbb)$ be a L\'evy process with the $\Fbb$-characteristics $(\bt,\sig\p2,\nu)$. For any $f\in L\p2(\nu)$, the following claims hold:

\textnormal{(i)} $(\X{f},\Fbb)$ is a L\'evy process and a true martingale. 

\textnormal{(ii)} For every $t\geq0$, the identity $\Ebb[(\X{f}_{\,t})\p{\,2}]=t\,\int_\Rbb f\p2\rmd \nu<+\infty$ holds. Hence, for every deterministic time $T\in\Rbb_+$, the process $(X\p{f}_{t\wedge T})_{t\geq0}$ belongs to $\Hscr\p2_0(\Fbb)$.

\textnormal{(iii)} $\pb{\X{f}}{\X{g}}_{\,t}=t\int_\Rbb fg\,\rmd \nu$, for every $f,g\in L\p{\,2}(\nu)$, $t\ge 0$.

\textnormal{(iv)} $\X{f}$ and $\X{g}$ are orthogonal martingales if and only if $f,g\in L\p{\,2}(\nu)$ are orthogonal functions.
\end{theorem}

If $\Tscr\subseteq L\p2(\nu)$, we set
\begin{equation}\label{eq:mart.fam}
\Xscr_\Tscr:=\{\Ws\}\cup\{\X{f},\  f\in\Tscr\}.
\end{equation} 
For Theorem \ref{thm:prp.lev} below we refer to \cite{DTE15}, \S4.2.
\begin{theorem}\label{thm:prp.lev} Let $\Tscr:=\{f_n,\ n\geq1\}\subseteq L\p2(\nu)$ be an orthonormal basis. Then the family $\Xscr_\Tscr$ possesses the \emph{PRP} with respect to $\Fbb\p L$, i.e., every $X\in\Hscr\p2(\Fbb\p L)$, can be written as
\begin{equation}\label{eq:prp.Fl}
X_t=X_0+Z\cdot\Ws_t+\sum_{n=1}\p\infty V\p n\cdot \X{f_n}_t,\qquad Z\in\Lrm\p2(\Fbb\p L,\Ws),\ V\p n\in\Lrm\p2(\Fbb\p L,X\p{f_n}),\  n\geq1, \  t\geq0.
\end{equation}

\end{theorem}

To simplify notation, we introduce the following space of processes:
\begin{definition}\label{def:M2}Let $\ell\p2$ denote the Hilbert space of sequences $v=(v\p{n})_{n\geq1}$ for which the norm $\|v\|_{\ell\p2}\p2:=\sum_{n=1}\p\infty (v\p{n})\p2$ is finite. 
We denote by $\M\p2(\Fbb,\ell\p2)$ the space of $\ell\p2$-valued $\Fbb$-predictable processes $V$ such that $\int_0\p {+\infty}\|V_s\|_{\ell\p2}\p2\rmd s$ is integrable. 
\end{definition}
We remark that $\sum_{n=1}\p\infty V\p n\cdot \X{f_n}\in\Hscr_0\p2(\Fbb)$ if and only if $V=(V\p{n})_{n\geq1}\in\M\p2(\Fbb,\ell\p2)$.

\subsection{Predictable Representation Property in Progressively Enlarged L\'evy Filtrations}\label{subs:prp.pr.en.Lev}

We now  consider a  L\'evy process $(L,\Fbb\p L)$ with $\Fbb\p L$-characteristics $(\bt,\sig\p2,\nu)$. For a random time $\tau$, $\Gbb=(\Gscr_t)_{t\ge0}$ denotes the progressive enlargement of $\Fbb\p L$ by $\tau$ (see \eqref{eq:def.Gtil}). 

Observe that, if $\tau$ fulfils Assumption \ref{ass:en.fi}, then $\Gscr_0$ is trivial, $\Fscr_0\p L$ being trivial (see Lemma \ref{lem:Z.avoi} (i)). 

We recall that the process $\Lm\p\Gbb$ was defined in \eqref{eq:G.com.gen}. Furthermore, $M:=H-\Lm\p\Gbb\in\Hscr_0\p2(\Gbb)$ and $\aPP{M}{M}:=\Lm\p\Gbb$ (see Theorem \ref{thm:G.com.AH} (ii)).

\begin{theorem}\label{theorem:orth.mart}
Let us consider a random time $\tau$ satisfying Assumption \ref{ass:en.fi} and a L\'evy process $(L,\Fbb\p L)$ with  $\Fbb\p L$-characteristics $(\bt,\sig\p2,\nu)$. Let $\Tscr\subseteq L\p2(\nu)$ be an orthonormal basis. Then, 

\textnormal{(i)} $(L,\Gbb)$ is a L\'evy process with $\Gbb$-characteristics $(\bt,\sig\p2,\nu)$.

\textnormal{(ii)} $\Xscr:=\Xscr_\Tscr\cup\{M\}$ is a family of $\Gbb$-martingales and $\Ebb[X_t\p2]<+\infty$,  $X\in\Xscr$, $t\geq0$.

\textnormal{(iii)} For every $X\in\Xscr_\Tscr$, the identity $[X,M]=0$ holds (up to an evanescent set). In particular, $\Xscr$ consists of pairwise orthogonal $\Gbb$-martingales.
\end{theorem}
\begin{proof}
Statement (i) immediately follows from hypothesis $(\Hscr)$ by an application of Lemma \ref{char.lev.proc}. Statement (ii) is a consequence of (i), Theorem \ref{thm:en.mar} with $\Fbb=\Gbb$ and Theorem \ref{thm:G.com.AH} (ii). We are now going to verify (iii). Let $\Tscr=\{f_n,\ n\geq1\}$. According to Theorem \ref{thm:G.com.AH} (i), the process $\Lm\p\Gbb$ is continuous. So, by the definition of $X\p{f_n}$ (see \eqref{eq:en.mar}) we get
\begin{equation}\label{eq:sq.Xf.H}
[\X{f_n},M]_t=[\X{f_n},H]_t=\sum_{s\leq t}f_n(\Delta L_s)\Delta H_s1_{\{\Delta L_s\neq0\}}=\sum_{s\leq t}f_n(\Delta L_s)\Delta H_s1_{\{\Delta L_s\neq0\}\cap\{\Delta H_s\neq0\}}.
\end{equation}
Since  $(L, \Fbb\p L)$ is a L\'evy process, there exists a sequence $(\et_n)_{n\geq1}$ of $\Fbb\p L$-stopping times with pairwise-disjoint graphs such that $\{\Delta L\neq0\}=\bigcup_{n=1}\p\infty[\et_n]$ (see \cite[Proposition I.1.32]{JS00}), where, for a stopping time $\et$, we denote by $[\et]$ the graph of $\et$. Due to (i), this decomposition of the random set $\{\Delta L\neq0\}$ also holds in $\Gbb$. Therefore, because of the identity $\{\Delta H\neq0\}=[\tau]$, we have $\{\Delta L\neq0\}\cap\{\Delta H\neq0\}=\bigcup_{n=1}\p\infty[\et_n]\cap[\tau]$. Because of hypothesis $(\Ascr)$, for every $n\geq1$, we have $\Pbb[\tau=\et_n<+\infty]=0$. Hence, the random sets $[\et_n]\cap[\tau]$, $n\ge1$, are evanescent. This means that $\{\Delta L\neq0\}\cap\{\Delta H\neq0\}$ is an evanescent random set. Therefore, \eqref{eq:sq.Xf.H} yields that $[\X{f_n},M]$ is indistinguishable from zero. Clearly, since by (i) $\Ws$ is a $\Gbb$-Brownian motion, the identity $[\Ws,M]=0$ also holds, $M$ being a $\Gbb$-purely discontinuous martingale.
This means, in particular, that the $\Gbb$-martingale $M$ is orthogonal to the family of pairwise-orthogonal $\Gbb$-martingales $\Xscr_\Tscr$. The proof of the theorem is complete. 
\end{proof}
Before we come to the main theorem of the present paper, we state a preliminary lemma.
\begin{lemma}\label{lem:prel.res}
Let $(\zt\p k)_{k\in\Nbb}$ be a sequence converging in $L\p2(\Om,\Gscr_{\infty},\Pbb)$ to $\zt$. If
\begin{equation}\label{eq:prel.res.seq}
\zt\p k=\Ebb[\zt\p k]+Z\p k\cdot\Ws_{\infty}+\sum_{n=1}\p\infty V\p{n,k}\cdot X\p{f_n}_{\infty}+U\p k\cdot M_{\infty}
\end{equation}
with $Z\p k\in \Lrm\p2(\Gbb,\Ws)$, $V\p k\in \M\p2(\Gbb,\ell\p2)$ and $U\p k\in \Lrm\p2(\Gbb,M)$, $k\in\Nbb$, 
then there exist $Z\in \Lrm\p2(\Ws,\Gbb)$, $V\in \M\p2(\Gbb,\ell\p2)$ and $U\in \Lrm\p2(\Gbb,M)$ such that 
\begin{equation}\label{eq:prel.res}
\zt=\Ebb[\zt]+Z\cdot\Ws_{\infty}+\sum_{n=1}\p\infty V\p{n}\cdot X\p{f_n}_{\infty}+U\cdot M_{\infty}.
\end{equation}
\end{lemma}
\begin{proof}
We set $\zt\p{1,k}:=Z\p k\cdot\Ws_\infty$, $\zt\p{2,k}:=\sum_{n=1}\p\infty V\p{n,k}\cdot X\p{f_n}_\infty$ and $\zt\p{3,k}:=U\p k\cdot M_\infty$. Then, for every $j$ and $k$, the estimate $\Ebb[(\zt\p{i,j}-\zt\p{i,k})\p2]\leq\Ebb[(\zt\p j-\zt\p{k})\p2]$ holds. So, $(\zt\p{i,k})_{k\in\Nbb}$ is a Cauchy sequence in $L\p2(\Om,\Gscr_\infty,\Pbb)$, for $i=1,2,3$. Let $\zt\p i$ be the limit in $L\p2(\Om,\Gscr_\infty,\Pbb)$ of $(\zt\p{i,k})_{k\in\Nbb}$, $i=1,2,3$. From the isometry, it follows that $(Z\p k)_{k\in\Nbb}$ is a Cauchy sequence in $\Lrm\p2(\Gbb,\Ws)$, $(V\p{k})_{k\in\Nbb}$ is a Cauchy sequence in $\M\p2(\Gbb,\ell\p2)$ and $(U\p k)_{k\in\Nbb}$ is a Cauchy sequence in $\Lrm\p2(\Gbb,M)$. By $Z$, $V$ and $U$ we denote the respective limit in the respective space of each of these Cauchy sequences and set $\et\p1:=Z\cdot\Ws_\infty$, $\et\p2:=\sum_{n=1}\p\infty V\p n\cdot X\p{f_n}_\infty$, $\et\p3:=U\cdot M_\infty$. Thanks to isometry, we get $\Ebb[(\zt\p i-\et\p i)\p2]=0$ and $\zt\p i=\et\p i$ a.s.,  $i=1,2,3$. Since $\zt=\zt\p1+\zt\p2+\zt\p3$, the proof is complete.
\end{proof}
We now prove the main result of this paper about the PRP with respect to the filtration $\Gbb$.
\begin{theorem}\label{thm:PRP.den.hp}
Let $(L,\Fbb\p L)$ be a L\'evy process with $\Fbb\p L$-characteristics $(\bt,\sig\p2,\nu)$. Denote by $\Tscr=\{f_n,\ n\geq1\}$ an orthonormal basis of $L\p2(\nu)$ and by $\Xscr_\Tscr\subseteq\Hlocz\p2(\Fbb\p L)$ the family of martingales associated with $\Tscr$ (see\ \eqref{eq:mart.fam}). Furthermore, given a random time  $\tau$ satisfying Assumption \ref{ass:en.fi}, let $\Gbb$ denote the progressive enlargement of $\Fbb\p L$ by $\tau$. 
Then, the orthogonal family of $\Gbb$-martingales $\Xscr=\Xscr_\Tscr\cup\{M\}\subseteq\Hlocz\p2(\Gbb)$ possesses the \emph{PRP} with respect to $\Gbb$, that is, every  $X\in \Hscr\p2(\Gbb)$ can be represented as
\begin{equation}\label{eq:rep.mar}
X_t=X_0+Z\cdot\Ws_t+\sum_{n=1}\p\infty V\p n\cdot \X{f_n}_t+U\cdot M_t,\quad t\geq0,
\end{equation}
where $Z\in\Lrm\p2(\Gbb,\Ws)$, $V\in\M\p2(\Gbb,\ell\p2)$ and $U\in \Lrm\p2(\Gbb,M)$. Moreover, this representation is unique, i.e., if there exists another triplet $Z\p\prime\in \Lrm\p2(\Gbb,\Ws)$, $V\p\prime\in\M\p2(\Gbb,\ell\p2)$, $U\p\prime\in \Lrm\p2(\Gbb,M)$ such that \eqref{eq:rep.mar} holds for $(Z\p\prime,V\p\prime, U\p\prime)$ instead of  $(Z,V, U)$, then we have
\[ \|Z-Z\p\prime\|_{\Lrm\p2(\Gbb,\Ws)}=0,\quad \|V-V\p\prime\|_{\M\p2(\Gbb,\ell\p2)}=0\quad \text{and}\quad \|U-U\p\prime\|_{\Lrm\p2(\Gbb,M)}=0.
\]
\end{theorem}
\begin{proof}
 First we observe that the uniqueness statement is an immediate consequence of the isometry induced by each one of the stochastic integrals on the right-hand side of \eqref{eq:rep.mar} between $\Hscr\p2(\Gbb)$ and the spaces $\Lrm\p2(\Gbb,\Ws)$, $\M\p2(\Gbb,\ell\p2)$ and $\Lrm\p2(\Gbb,M)$, respectively. 

We now show the representation formula \eqref{eq:rep.mar}. More precisely, as an application of the monotone class theorem, we are going to show that every $\zt\in L\p2(\Om,\Gscr_\infty,\Pbb)$ can be represented as
\begin{equation}\label{eq:rep.aux}
\zt=\Ebb[\zt]+Z\cdot\Ws_{\infty}+\sum_{n=1}\p\infty V\p{n}\cdot X\p{f_n}_{\infty}+U\cdot M_{\infty}.
\end{equation}
Since $(\Hscr\p2(\Gbb),\|\cdot\|_{\Hscr\p2(\Gbb)})$ and $(L\p2(\Om,\Gscr_\infty,\Pbb),\|\cdot\|_2)$ are isomorphic, where $\|\cdot\|_2$ denotes the usual $L\p2$-norm, \eqref{eq:rep.aux} is equivalent to \eqref{eq:rep.mar}.

To begin with, we observe that the system
\[
\Cscr:=\big\{\zt=\xi(1-H_{s}),\quad \xi\ \ \Fscr_\infty\p L\textnormal{-measurable and bounded, }\ s\geq0\big\}
\]
generates $\Gscr_\infty$. Indeed, we clearly have the inclusion $\sig(\Cscr)\subseteq\Gscr_\infty$. Conversely, taking $s=0$ we have $H_0=0$ and hence $\Fscr_\infty\p L\subseteq\sig(\Cscr)$. Additionally, taking $\xi=1$, we get $H_s=1-(1-H_s)$, for every $s\geq0$, and hence $\Hscr_\infty\subseteq\sig(\Cscr)$. Consequently, $\Gscr_\infty=\Fscr^L_\infty\vee\Hscr_\infty\subseteq \sigma(\Cscr)$. Furthermore, the system $\Cscr$ is stable under multiplication. 

As a first step, we verify \eqref{eq:rep.aux} for random variables in $\Cscr$. Let $s\geq0$ be arbitrary but fixed. Recalling the definition of $A$ in \eqref{eq:az.supm}, we define the random variable $\xi\p\prime:=\xi A_s$, where $\xi$ is a bounded and $\Fscr\p L_\infty$-measurable random variable. As $\xi\p\prime$ is bounded, the $\Gbb$-martingale $X\p\prime$ satisfying $X\p\prime_t=\Ebb[\xi\p\prime|\Gscr_t]$ a.s., $t\geq0$, is bounded. By hypothesis $(\Hscr)$ and Proposition \ref{prop:imm.eq.for} (iii), $\xi\p\prime$ being $\Fscr_\infty\p L$-measurable, we have $\Ebb[\xi\p\prime|\Gscr_t]=\Ebb[\xi\p\prime|\Fscr\p L_t]$ and hence $X\p\prime$ belongs to $\Hscr\p2(\Fbb\p L)$. Therefore, $X\p\prime$ can be represented as in \eqref{eq:prp.Fl}. Since $Y:=\Escr(-M)$ satisfies the Dol\'eans-Dade equation with respect to $-M$, by \eqref{eq:prp.Fl}, the properties of the quadratic covariation and Theorem \ref{theorem:orth.mart} (iii), we get
\begin{equation}\label{eq:cov.Xpr.Y} 
[X\p\prime,Y]=\Big[X_0+Z\cdot\Ws+\sum_{n=1}\p\infty V\p n\cdot \X{f_n},1-Y_-\cdot M\Big]=-\sum_{n=1}\p\infty V\p nY_-\cdot[\X{f_n}, M]=0.
\end{equation} 
 Furthermore, Theorem \ref{thm:G.com.AH} (iii) yields
\begin{equation}\label{eq:aux.rep}
\xi(1-H_s)=\xi1_{[0,\tau)}(s)=\xi1_{[0,\tau)}(s)A_sA\p{-1}_s=\xi\p\prime Y_s=X\p\prime_\infty Y\p s_\infty,
\end{equation}
where  $Y\p s$ denotes the process $Y$ stopped at $s$, i.e., $Y\p s_t:=Y_{t\wedge s}$. We stress that \eqref{eq:aux.rep} is well-posed because, in view of Theorem \ref{thm:con.Z} (iii), we have $\ A>0$ on $[0,\tau]$. 

We now denote by $X$ the bounded $\Gbb$-martingale satisfying $X_t=\Ebb[\xi(1-H_s)|\Gscr_t]$ a.s., $t\ge0$. According to \eqref{eq:aux.rep}, using the formula of integration by parts and \eqref{eq:cov.Xpr.Y}, we compute
\begin{equation}\label{eq:rep.C}
\begin{split}
X_\infty&=\xi(1-H_s)
\\&=X\p\prime_0+Y_-\p s\cdot X\p\prime_\infty+X\p\prime_-1_{[0,s]}\cdot Y_\infty
\\&=X\p\prime_0+Y_-\p sZ\cdot\Ws_\infty+\sum_{n=1}\p{+\infty}Y_-\p sV\p n\cdot\X{f_n}_\infty-X\p\prime_-Y_-1_{[0,s]}\cdot M_\infty,
\end{split}
\end{equation}
where, in the last identity, we again used \eqref{eq:prp.Fl} for $X\p\prime$ and the properties of the stochastic integral.

Because $X$ is a bounded martingale, $\aPP{X}{X}_\infty$ is integrable. We now set $X\p1:=Y_-\p sZ\cdot\Ws$, $X\p2:=\sum_{n=1}\p{+\infty}Y_-\p sV\p n\cdot\X{f_n}$ and $X\p3:=Y_-1_{[0,s]}\cdot M$. Then, in view of Theorem \ref{theorem:orth.mart} (ii), $X\p i$ belongs to $\Hlocz\p2(\Gbb)$, $i=1,2,3$.  Theorem \ref{theorem:orth.mart} (iii) and \eqref{eq:rep.C} now yield
\[
\aPP{X}{X}_\infty=\aPP{X\p 1}{X\p 1}_\infty+\aPP{X\p 2}{X\p 2}_\infty+\aPP{X\p 3}{X\p 3}_\infty.
\]
Hence, $\aPP{X\p i}{X\p i}_\infty\leq\aPP{X}{X}_\infty$ is integrable and $X\p i\in\Hscr_0\p2(\Gbb)$ by Lemma \ref{lem:loc.sq.in}, $i=1,2,3$. As a consequence of \cite[Proposition 2.48 (b)]{J79}, the properties of $\M(\Gbb,\ell\p2)$, and Theorem \ref{thm:en.mar} (iii), we deduce 
\[
Y_-\p sZ\in \Lrm\p2(\Gbb,\Ws),\quad Y_-\p sV\in\M\p2(\Gbb,\ell\p2),\quad X\p\prime_-Y_-1_{[0,s]}\in\Lrm\p2(\Gbb,M).
\]
The proof of \eqref{eq:rep.aux} for elements in $\Cscr$ is now complete.

Let us denote by $\Kscr$ the class of all bounded $\Gscr_\infty$-measurable random variables $\zt$ which can be represented as in  \eqref{eq:rep.aux}. Clearly, because of the previous step, the inclusion $\Cscr\subseteq\Kscr$ holds. Furthermore, $\Kscr$ is a monotone class of $\Gscr_\infty$-measurable random variables. To see this, we consider $(\zt\p k)_{k\in\Nbb}\subseteq\Kscr$ such that $0\leq\zt\p k\leq\zt\p{k+1}\leq c$, $c>0$, and $\zt\p k\uparrow\zt$, $k\rightarrow+\infty$. So $\zt$ is $\Gscr_\infty$-measurable and bounded. We have to show that $\zt$ belongs to $\Kscr$. But this immediately follows from Lemma \ref{lem:prel.res}, observing that, as an application of Lebesgue's dominated convergence theorem, $(\zt\p k)_{k\in\Nbb}$ converges to $\zt$ also in $L\p2(\Om,\Gscr_\infty,\Pbb)$. By the monotone class theorem for functions (see \cite[Theorem 1.4]{HWY92}), we can conclude that $\Kscr$ contains all bounded $\Gscr_\infty$-measurable random variables.

Now, using Lemma \ref{lem:prel.res} again, it is a standard procedure to obtain \eqref{eq:rep.aux} for every  $\zt\in L\p2(\Om,\Gscr_\infty,\Pbb)$, $\zt\geq0$: Indeed, $\zt\p k:=\zt\wedge k$ is bounded and, by Lebesgue's dominated convergence theorem, $(\zt\p k)_{k\geq1}$ converges to $\zt$ in $L\p2(\Om,\Gscr_\infty,\Pbb)$. Because of the previous step, $\zt\p k$ can be represented as in \eqref{eq:prel.res.seq} and therefore, from Lemma \ref{lem:prel.res}, $\zt$ has the representation \eqref{eq:rep.aux}. For an arbitrary $\zt\in L\p2(\Om,\Gscr_\infty,\Pbb)$, we have $\zt=\zt\p +-\zt\p -$ and $\zt\p\pm\in L\p2(\Om,\Gscr_\infty,\Pbb)$, $\zt\p\pm\geq0$. So, by the previous step and by the linearity of the stochastic integral, $\zt$ can be represented as in \eqref{eq:rep.aux}. The proof of the theorem is complete.
\end{proof}
Lemma \ref{lem:prel.res} and Theorem \ref{thm:PRP.den.hp} are formulated for the progressively enlarged L\'evy filtration $\Fbb\p L$ and the family $\Xscr_\Tscr$, where $\Tscr\subseteq L\p2(\nu)$ is an orthonormal basis. However, for any filtration $\Fbb$ satisfying the usual conditions, the proofs of Lemma \ref{lem:prel.res} and Theorem \ref{thm:PRP.den.hp} work without any difference also for any countable family $\Xscr=\{X\p n,\ n\geq1\}\subseteq\Hscr\p2(\Fbb)$ of mutually orthogonal $\Fbb$-martingales possessing the PRP with respect to $\Fbb$. So, the following more general result holds:

\begin{theorem}\label{thm:PRP.den.hp.gen}
Let $\Fbb$ be a filtration satisfying the usual conditions. Suppose that $\Xscr=\{X\p n,\ n\geq1\}\in\Hscr\p2(\Fbb)$\footnote{It is actually sufficient that $\Xscr$ is a family of martingales such that $\Ebb[(X\p n_t)\p2]<+\infty$, for every $t\geq0$, $n\geq1$.} is an arbitrary family of mutually orthogonal $\Fbb$-martingales possessing the \emph{PRP} with respect to $\Fbb$. Consider a random time $\tau$ meeting Assumption \ref{ass:en.fi}, and let $\Gbb$ be the progressive enlargement of $\Fbb$ by $\tau$. Then,

\textnormal{(i)} $\Xscr$ is an orthogonal family of square integrable $\Gbb$-martingales. 

\textnormal{(ii)} For every $n\geq1$, the identity $[X\p n,M]=0$ holds. Hence $X\p n,\ M\in\Hscr\p2(\Fbb)$ are orthogonal.  

\textnormal{(iii)} Every  $X\in \Hscr\p2(\Gbb)$ can be represented as
\[
X_t=X_0+\sum_{n=1}\p\infty V\p n\cdot X\p{n}_t+U\cdot M_t,\quad t\ge0,
\]
where  $V\p n\in\Lrm\p2(\Gbb,X\p n)$, $n\geq1$, and $U\in \Lrm\p2(\Gbb,M)$. Furthermore, $V\p n$ and $U$ are $\aPP{X\p n}{X\p n}\otimes\Pbb$-a.e.\ and $\aPP{M}{M}\otimes\Pbb$-a.e.\ unique on $\Rbb_+\times\Om$, respectively.
\end{theorem}

\section{The multiplicity of a progressively enlarged filtration}\label{subs:sp.c.bf}
In this section we study the multiplicity of the progressively enlarged Brownian filtration. More precisely, for a Brownian motion $W\p\sig$ we consider the progressive enlargement $\Gbb$ of  $\Fbb\p{W\p\sig}$ by a random time $\tau$ satisfying Assumption \ref{ass:en.fi}. Because of Assumption \ref{ass:en.fi} and Theorem \ref{thm:PRP.den.hp}, $M=H-\Lm\p\Gbb$ and $\Ws$ are orthogonal $\Gbb$-local martingales, $\Ws$ being continuous and $M$ purely discontinuous, and the family $\{\Ws,M\}$ has the PRP with respect to $\Gbb$. Hence, one could be tempted to think that the multiplicity of $\Gbb$ (see Definition \ref{def:mult} below) is equal to two. However, as we are going to see now,  this is not the case: Surprisingly, it can happen that the multiplicity of $\Gbb$ is equal to one.

We stress that the propagation of the PRP of a Brownian motion to the progressively enlarged Brownian filtration has been studied by Kusuoka in \cite{Ku99} under some conditions which are stronger than those of Theorem \ref{thm:PRP.den.hp}. In particular, Kusuoka requires in \cite{Ku99} that $\tau$ satisfies hypothesis $(\Hscr)$ and that the $\Gbb$-compensator $\Lm\p\Gbb$ of $\tau$ is absolutely continuous with respect to the Lebesgue measure. In Proposition \ref{prop:avoid.bf} below, we show that this latter assumption on $\Lm\p\Gbb$ implies hypothesis $(\Ascr)$, i.e., that the assumptions of Kusuoka imply our  Assumption \ref{ass:en.fi}. As we are going to prove, the multiplicity of $\Gbb$ reduces to one if and only if $\Lm\p\Gbb$ is singular continuous. Therefore, the study of the multiplicity of $\Gbb$ presented in this part is only possible because of our main result, that is, Theorem \ref{thm:PRP.den.hp}, and cannot be obtained from the results by Kusuoka in \cite{Ku99}.  

Finally, we notice that for this part we need some general properties of stochastic integrals which are of independent interest, not only in view of the theory of enlargement of filtrations. These properties are therefore discussed in Subsection \ref{subs:gen.con} in a general context. Then, in Subsection \ref{subs:mul.bf}, the results of Subsection \ref{subs:gen.con} are applied to study, in particular, the multiplicity of the enlarged Brownian filtration.

\subsection{General considerations}\label{subs:gen.con}

Let $\Fbb$ denote an arbitrary filtration satisfying the usual conditions. To begin with, we recall the definition of the multiplicity of $\Fbb$, for the first time introduced by Davis and Varaiya in \cite{DV74}. 
\begin{definition}\label{def:mult}
The multiplicity $n$ of $\Fbb$ is defined as the minimal number of orthogonal locally square integrable $\Fbb$-local martingales $X\p1,\ldots,X\p n$ which is necessary to represent each element in $\Hscr\p2(\Fbb)$ as an orthogonal sum of stochastic integrals with respect to $X\p1,\ldots,X\p n$.
\end{definition}

We now assume that $\Fbb$ supports two orthogonal local martingales $X,Y\in\Hscr_{\mathrm{loc}}\p2(\Fbb)$ such that the family $\{X,Y\}$ possesses the PRP with respect to $\Fbb$. We are going to establish necessary and sufficient conditions on $X$ and $Y$ for the filtration $\Fbb$ to have multiplicity equal to one. 

The following definition is similar to the one of mutually singular measures on $\Pscr(\Fbb)$ given in \cite[Lemma 6]{WG82}. 
\begin{definition}\label{def:sing.meas} Let $X,Y\in\Hscr\p2_\mathrm{loc}(\Fbb)$.
We say that $\aPP{X}{X}$ and $\aPP{Y}{Y}$ are mutually singular on $\Pscr(\Fbb)$, if there exist a set $D\in\Pscr(\Fbb)$ such that $\aPP{X}{X}=1_{D}\cdot\aPP{X}{X}$ and $\aPP{Y}{Y}=1_{D\p c}\cdot\aPP{Y}{Y}$ (up to an evanescent set).
\end{definition} 
We now proceed with a general property of stochastic integrals with respect to a locally square integrable local martingale.
\begin{lemma}\label{lem:converse}
Let $N\in\Hscr\p2_{\mathrm{loc}}(\Fbb)$. Define $X:=U\cdot N$ and $Y:=V\cdot N$, for $U,V\in\Lrm_{\mathrm{loc}}\p2(\Fbb,N)$. Then, $X,Y\in\Hscr\p2_\mathrm{loc}(\Fbb)$ and, moreover, $X$ is orthogonal to $Y$ if and only if $\aPP{X}{X}$ and $\aPP{Y}{Y}$ are mutually singular on $\Pscr(\Fbb)$. 
\end{lemma}
\begin{proof}
By localization, we can assume that $N\in\Hscr\p2(\Fbb)$ and $U,V\in\Lrm\p2(\Fbb,N)$. Hence, we have $X,Y\in\Hscr\p2(\Fbb)$. We denote by $m_X$, $m_Y$ and $m_{XY}$ the measures on $(\Rbb_+\times\Om,\Pscr(\Fbb))$ defined by
\[
m_X(B)=\Ebb[1_B\cdot\aPP{X}{X}_\infty],\quad m_Y(B)=\Ebb[1_B\cdot\aPP{Y}{Y}_\infty]\quad\text{and}\quad m_{XY}(B)=\Ebb[1_B\cdot\Var(\aPP{X}{Y})_\infty],
\]
for $B\in\Pscr(\Fbb)$, $\Var(\aPP{X}{Y})$ denoting the total-variation  process of $\aPP{X}{Y}$. Notice that  the identities $\aPP{X}{X}=U\p2\cdot\aPP{N}{N}$, $\aPP{Y}{Y}=V\p2\cdot\aPP{N}{N}$ and $\Var(\aPP{X}{Y})=|UV|\cdot\aPP{N}{N}$ hold.

We now assume that $\aPP{X}{X}$ and $\aPP{Y}{Y}$ are mutually singular on $\Pscr(\Fbb)$ and show that $\aPP{X}{Y}=0$ holds. For this, we consider $D\in\Pscr(\Fbb)$ such that 
\[
\aPP{X}{X}=1_D\cdot\aPP{X}{X},\quad\aPP{Y}{Y}=1_{D\p c}\cdot\aPP{Y}{Y}.
\]
 We then have
\[
U\p2=\frac{\rmd\, m_X}{\rmd\,\aPP{N}{N}\otimes\Pbb}=1_DU\p2,\quad V\p2=\frac{\rmd\, m_Y}{\rmd\,\aPP{N}{N}\otimes\Pbb}=1_{D\p c}V\p2,\quad |UV|=\frac{\rmd\, m_{XY}}{\rmd\,\aPP{N}{N}\otimes\Pbb}. 
\]

By the uniqueness of the densities, we get $|U|=1_{D}|U|$ and $|V|=1_{D\p c}|V|$ $\aPP{N}{N}\otimes\Pbb$-a.e. Therefore, we deduce $m_{XY}(B)=0$ for every $B\in\Pscr(\Fbb)$. In particular, we can choose $B=\Rbb_+\times\Om$, which yields $\Var(\aPP{X}{Y})_\infty=0$ a.s. Therefore, $\Var(\aPP{X}{Y})=0$ and hence $\aPP{X}{Y}=0$ (up to an evanescent set).

Conversely, we now assume that $\aPP{X}{Y}=0$ holds. 
By the definition of $m_{XY}$, this yields $|UV|=0$ $\aPP{N}{N}\otimes\Pbb$-a.e. We consider the set $D=\{(t,\om)\in\Rbb_+\times\Om:U_t(\om)\neq0\}$ which clearly satisfies $D\in\Pscr(\Fbb)$, $U$ being predictable. We then have $\aPP{X}{X}=1_D\cdot\aPP{X}{X}$. Since $1_DV=0$ $\aPP{N}{N}\otimes\Pbb$-a.e., $m_Y(1_{D\cap B})=0$ holds, for every $B\in\Pscr(\Fbb)$. So, introducing the $\Fbb$-predictable \cadlag increasing and integrable process $K=1_{D\p c}\cdot\aPP{Y}{Y}=1_{D\p c}V\p2\cdot\aPP{N}{N}$, we see that $K$ satisfies $K_0=0$ and $\Ebb[1_B\cdot K_\infty]=\Ebb[1_B\cdot\aPP{Y}{Y}_\infty]$ for every $B\in\Pscr(\Fbb)$. This yields $\Ebb[ K_\et]=\Ebb[\aPP{Y}{Y}_\et]$ for every $\Fbb$-stopping time $\et$ (take $B=[0,\et]$). But then, by the predictability of $\aPP{Y}{Y}$ and the uniqueness of the dual predictable projection, $K$ and $\aPP{Y}{Y}$ are indistinguishable. The proof of the lemma is complete.
\end{proof}
We are now ready to prove the following result.
\begin{theorem}\label{thm:prp.sum}
Let $X,Y\in\Hscr\p2_{\mathrm{loc}}(\Fbb)$ be orthogonal and such that the family $\{X,Y\}$ possesses the \emph{PRP} with respect to $\Fbb$. Then, the following statements are equivalent

\textnormal{(i)} The multiplicity of the filtration $\Fbb$ is equal to one.

\textnormal {(ii)} The processes $\aPP{X}{X}$ and $\aPP{Y}{Y}$ are mutually singular on $\Pscr(\Fbb)$.

\textnormal{(iii)} The local martingale $Z\in\Hscr\p2_{\mathrm{loc}}(\Fbb)$ defined by $Z:=X+Y$ possesses the \emph{PRP} with respect to $\Fbb$.
\end{theorem}
\begin{proof}
We first observe that from (iii) we immediately get (i), by the definition of the multiplicity of a filtration. 

Next we show the implication (i)$\Rightarrow$(ii). If the multiplicity of $\Fbb$ is equal to one, then there exists an $N\in\Hscr\p2_\mathrm{loc}(\Fbb)$ possessing the PRP with respect to $\Fbb$. Hence, there exist $U,V\in\Lrm_\mathrm{loc}\p2(\Gbb,N)$ such that $X-X_0=U\cdot N$ and $Y-Y_0=V\cdot N$. So, (ii) follows from the orthogonality of $X$ and $Y$ and from Lemma \ref{lem:converse}, since we have $\aPP{X-X_0}{X-X_0}=\aPP{X}{X}$ and $\aPP{Y-Y_0}{Y-Y_0}=\aPP{Y}{Y}$ because of the properties of the predictable covariation. 

Finally, we verify the implication (ii)$\Rightarrow$(iii). By localization, we can assume $X,Y\in\Hscr\p2(\Fbb)$. Hence, $Z\in\Hscr\p2(\Fbb)$. Using the PRP of $\{X,Y\}$, we know that every $R\in\Hscr\p2(\Fbb)$ can be represented as
\begin{equation}\label{eq:rep.R}
R=R_0+U\cdot X+V\cdot Y,\quad U\in\Lrm\p2(\Fbb,X),\quad V\in\Lrm\p2(\Fbb,Y).
\end{equation}
Due to the assumption (ii), there exists $D\in\Pscr(\Fbb)$ such that $\aPP{X}{X}=1_D\cdot\aPP{X}{X}$ and $\aPP{Y}{Y}=1_{D\p c}\cdot\aPP{Y}{Y}$. We then define $G:=1_DU+1_{D\p c}V$ and get $G\in\Lrm\p2(\Gbb,Z)$. Indeed, we have $G\p2=1_DU\p2+1_{D\p c}V\p2$ and $\aPP{Z}{Z}=\aPP{X}{X}+\aPP{Y}{Y}$. Hence,
\[
\Ebb[G\p2\cdot\aPP{Z}{Z}_\infty]=\Ebb[U\p2\cdot\aPP{X}{X}_\infty]+\Ebb[V\p2\cdot\aPP{Y}{Y}_\infty]<+\infty.
\]
But now we observe that $G\cdot Z$ and $U\cdot X+V\cdot Y$ are indistinguishable. Indeed, using the isometry, the orthogonality of $X$ and $Y$, and the mutual singularity of $\aPP{X}{X}$ and $\aPP{Y}{Y}$, we get
\[
\Ebb\big[(G\cdot Z_\infty-U\cdot X_\infty-V\cdot Y_\infty)\p2\big]=0.
\]
From this latter identity it follows that the martingales $G\cdot Z$ and $U\cdot X+V\cdot Y$ are modifications of each other. Since they are both c\`adl\`ag, they are actually indistinguishable. Combining this with \eqref{eq:rep.R}, we get the PRP for $Z$. The proof of the theorem is complete.
\end{proof}
\subsection{The multiplicity of the progressively enlarged Brownian filtration}\label{subs:mul.bf}
We now apply the results of Subsection \ref{subs:gen.con} to study the multiplicity of a progressively enlarged Brownian filtration. We assume that $\Ws$ is an $\Fbb\p{\Ws}$-Brownian motion, i.e. a L\'evy process with characteristic triplet $(0,\sig\p2,0)$. It is then well-known that every $R\in\Hscr\p2(\Fbb\p{\Ws})$ has the representation
\[
R_t=R_0+U\cdot \Ws_t,\quad U\in\Lrm\p2(\Fbb\p{\Ws},\Ws).
\]  
The problem of the propagation of the martingale representation theorem to the progressively enlarged filtration $\Gbb$ in this special case has been investigated by Kusuoka in \cite{Ku99} for a finite-valued random time $\tau$ such that $\Fbb\p L$ is immersed in $\Gbb$ and that the $\Gbb$-predictable compensator $\Lm\p\Gbb$ of $\tau$ is of the form
\[
\Lm\p\Gbb_t=\int_0\p{t\wedge\tau}\lm_s\rmd s,\quad t\geq0,
\]
where $\lm$ is a nonnegative Lebesgue-integrable and, without loss of generality, a $\Fbb\p L$-predictable process. The next proposition shows that the result of Kusuoka (see \cite[Theorem 2.3]{Ku99}) is a special case of our main result, Theorem \ref{thm:PRP.den.hp}.

\begin{proposition}\label{prop:avoid.bf}
Let $\Ws$ be a Brownian motion and consider a random time $\tau$ such that $\Fbb\p{\Ws}$ is immersed in $\Gbb$. Then, the $\Gbb$-predictable compensator $\Lm\p\Gbb$ of $\tau$ is continuous if and only if $\tau$ avoids $\Fbb\p{\Ws}$-stopping times.
\end{proposition}
\begin{proof}
Let $\Lm\p\Gbb$ be continuous. First, we observe that, since all $\Fbb\p {\Ws}$-martingales are continuous, from \cite[Corollary IV.5.7]{RY05}, every $\Fbb\p {\Ws}$-optional process is predictable. Let now $\et$ be a finite-valued $\Fbb\p {\Ws}$-stopping time. Then, $1_{[\et]}$ is an $\Fbb\p {\Ws}$-predictable process and hence $1_{[\et]}$ is $\Gbb$-predictable. 
As a consequence of the properties of the dual predictable projection, we get
\begin{equation}\label{eq:av.st.t.bm}
0=\Ebb\big[\Delta\Lm\p\Gbb_\et\big]=\Ebb\big[1_{[\et]}\cdot\Lm\p\Gbb_\infty\big]=\Ebb\big[1_{[\et]}\cdot H_\infty\big]=\Ebb\big[\Delta H_\et\big]=\Pbb\big[\tau=\et\big],
\end{equation}
for every finite-valued $\Fbb\p {\Ws}$-stopping time $\et$. From this it immediately follows that $\tau$ avoids $\Fbb\p {\Ws}$-stopping times. 

Conversely, we now assume that $\tau$ avoids $\Fbb\p {\Ws}$-stopping times. Then, the continuity of $\Lm\p\Gbb$ follows from \eqref{eq:G.com.gen} and and Lemma \ref{lem:Z.avoi} (ii). The proof is complete.
\end{proof}

Because of Theorem \ref{thm:PRP.den.hp}, we can consider the case of a random time $\tau$ satisfying hypothesis $(\Hscr)$ and such that the compensator $\Lm\p\Gbb$ is continuous and singular with  respect to the Lebesgue measure. To obtain  such a random time $\tau$, that by Proposition \ref{prop:avoid.bf} satisfies Assumption \ref{ass:en.fi}, one can proceed in a standard way following the \emph{Cox-construction} described in \cite[Section 6.5]{BR13}.

As an example, we can assume that $\Lm\p\Gbb$ is absolutely continuous with respect to the Devil's staircase on $\Rbb_+$, denoted by $C$ (in particular, $C_0=0$ and $C_\infty=\infty$). Since $\Lm\p\Gbb=(\Lm\p\Gbb)\p\tau$ and $C$ is a deterministic process, by \cite[Proposition I.3.13]{JS00} there exists a $\Gbb$-predictable nonnegative process $\kappa$ such that
\begin{equation}\label{eq:dev.rt}
\Lm\p\Gbb_t=\int_0\p{t\wedge\tau}\kappa_t\rmd C_t, \quad t\geq0.
\end{equation} 
Because of \cite[Lemme 4.4 (b)]{Jeu80} (see also \cite[Proposition 2.11 (b)]{AJ17}), one can assume, without loss of generality, that $\kappa$ is $\Fbb\p {\Ws}$-predictable.

\begin{theorem}\label{thm:mul.G}
Let $(\Ws,\Fbb\p{\Ws})$ be a Brownian motion. Denote by $\Gbb$ the progressive enlargement of $\Fbb\p{\Ws}$ by a random time $\tau$ that fulfils Assumption \ref{ass:en.fi}.
Then, the multiplicity of $\Gbb$ is equal to one if and only if $\Lm\p\Gbb$ is singular continuous on $\Pscr(\Gbb)$, that is, if and only if the $\Gbb$-martingale $Z:=\Ws+M$  possesses the \emph{PRP} with respect to $\Gbb$.
\end{theorem}
\begin{proof}
Theorem \ref{thm:G.com.AH} (ii) yields $M\in\Hscr\p2_0(\Gbb)$ and $\aPP{M}{M}=\Lm\p\Gbb$. Furthermore, Theorem \ref{theorem:orth.mart} applied to $L=\Ws$ implies that $\Ws$ is a $\Gbb$-Brownian motion such that $\aPP{\Ws}{\Ws}_t=\sig\p2 t$, $t\geq0$, and $\aPP{\Ws}{M}=0$. Finally, Theorem \ref{thm:PRP.den.hp} ensures that the family $\{\Ws,M\}$ has the PRP with respect to $\Gbb$. The claim now immediately follows from Theorem \ref{thm:prp.sum} applied to $\Fbb=\Gbb$, $X=\Ws$ and $Y=M$. The proof is complete.
\end{proof}

Using Theorem \ref{thm:PRP.den.hp.gen} instead of Theorem \ref{thm:PRP.den.hp}, we can show the following more general result. 
\begin{theorem}\label{thm:more.gen}
Let $\Fbb$ be an arbitrary filtration satisfying the usual conditions. Assume that $X\in\Hloc\p2(\Fbb)$ possesses the \emph{PRP} with respect to $\Fbb$. For a random time $\tau$ meeting Assumption \ref {ass:en.fi}, denote by $\Gbb$ the progressive enlargement of $\Fbb$ by $\tau$. 
Then, the multiplicity of $\Gbb$ is equal to one if and only if $\aPP{X}{X}$ and $\Lm\p\Gbb$ are mutually singular on $\Pscr(\Gbb)$, that is, if and only if the $\Gbb$-local martingale $Z:=X+M\in\Hscr\p2_{\mathrm{loc}}(\Gbb)$  possesses the \emph{PRP} with respect to $\Gbb$.
\end{theorem}
\begin{proof}
Theorem \ref{thm:G.com.AH} (ii) yields $M\in\Hscr\p2_0(\Gbb)$ and $\aPP{M}{M}=\Lm\p\Gbb$. The immersion property implies that $\aPP{X}{X}$ does not change in $\Gbb$. Thanks to Theorem  \ref{thm:PRP.den.hp.gen}, $\{X,M\}\subseteq\Hscr\p2_\mathrm{loc}(\Gbb)$ is a family of orthogonal $\Gbb$-local martingales possessing the PRP with respect to $\Gbb$. The claim now immediately follows from Theorem \ref{thm:prp.sum} applied to the filtration $\Gbb$ and with $Y=M$. The proof is complete.
\end{proof}
We stress that, if $\Fbb$ is more general than a Brownian filtration and supports martingales with jumps, then the continuity of $\Lm\p \Gbb$ does not imply hypothesis $(\Ascr)$. However, to construct a random time $\tau$ satisfying the assumptions of Theorem \ref{thm:more.gen}, we can follow the Cox-method (see Remark \ref{rem:Cox.con}). That is, also in this more general context, there exist random times $\tau$ satisfying Assumption \ref{ass:en.fi} and such that $\Lm\p\Gbb$ can be chosen singular on $\Pscr(\Gbb)$. For example, if $\aPP{X}{X}$ is absolutely continuous with respect to the Lebesgue measure, we can consider $\tau$ such that the compensator $\Lm\p\Gbb$ is given as in \eqref{eq:dev.rt} with an $\Fbb$-predictable density $\kappa$.
  
\begin{example}[An application of Theorem \ref{thm:more.gen}]\label{ex.app.th.ch} As an application of Theorem \ref{thm:more.gen}, we can consider the following example: Let $W$ be a standard Brownian motion and let $C$ denote the Devil's staircase on $\Rbb_+$. Since $C$ is increasing and continuous, $C$ is a \emph{continuous time change}. We define the process $X=(X_t)_{t\geq0}$ by $X_t=W_{C_t}$, $t\geq0$. Then,  by the properties of continuous time changes, $(X,\Fbb\p X)$ is a continuous local martingale and $\aPP{X}{X}_t=C_t$, $t\ge0$. From \cite[Theorem II.4.4]{JS00} we obtain that $X$ is  a non-homogeneous Brownian motion with respect to $\Fbb=\Fbb\p X$. If we now enlarge $\Fbb\p X$ to $\Gbb$ by a random time $\tau$ satisfying Assumption \ref{ass:en.fi} and such that $\Lm\p\Gbb$ is \emph{absolutely continuous} with respect to the Lebesgue measure, then Theorem \ref{thm:more.gen} implies that $Z=X+M$ has the PRP with respect to $\Gbb$. Therefore, the multiplicity of $\Gbb$ is equal to one. We can proceed in a similar way if, instead of $W$, we start with a homogeneous Poisson process $N$ and define $X_t=N_{C_t}-C_t$. More generally, outside of the context of processes with independent increments, we can also consider a point process $N$ such that the predictable compensator $N\p p$ is absolutely continuous with respect to the Lebesgue measure and set $X_t=N_{C_t}-N\p p_{C_t}$.
\end{example}
\paragraph*{Acknowledgement.} PDT gratefully acknowledges Martin Keller-Ressel and funding from the German Research Foundation (DFG) under grant ZUK 64.  

\end{document}